\documentclass[11pt]{article}
\usepackage[a4paper,margin=3cm]{geometry}
\usepackage[T1]{fontenc}      
\usepackage{lmodern}          
\usepackage{amsmath}          
\usepackage{amssymb}          
\usepackage{amsthm}           
\usepackage{mathtools}        
\usepackage{tikz}             
\usepackage{placeins}         
\usepackage[hidelinks]{hyperref} 
\hypersetup{
  pdftitle={Counting spanning quasi-trees of ribbon graphs:
            determinants and \#P-completeness},
  pdfauthor={William Whistler}}

\newtheorem{theo}{Theorem}[section]
\newtheorem{lemm}[theo]{Lemma}
\newtheorem{coro}[theo]{Corollary}
\newtheorem{obse}[theo]{Observation}
\newtheorem{conj}[theo]{Conjecture}
\newtheorem*{theoA}{Theorem A}
\newtheorem*{theoB}{Theorem B}
\newtheorem*{theoC}{Theorem C}
\newtheorem*{theoD}{Theorem D}

\newcommand{\bc}{\operatorname{bc}}               
\newcommand{\qtn}{N}                              
\newcommand{\eg}{\varepsilon}                     
\newcommand{\crx}[1]{n_{\times}(#1)}              
\newcommand{\corner}[2]{(#1,\mathrm{#2})}         
\newcommand{\fim}{A}                              
\newcommand{\lcg}{H}                              
\newcommand{\skl}{\widetilde{A}}                  
\newcommand{\GFtwo}{\mathbb{F}_{2}}
\newcommand{\dettwo}{{\det}_{2}}                  
\newcommand{\nulltwo}{\operatorname{null}_{2}}    
\newcommand{\rktwo}{\operatorname{rk}_{2}}        

\newcommand{\ip}{q}                               
\newcommand{\ipn}{q_{N}}                          
\newcommand{\fourmap}{F}                          
\newcommand{\atrail}{C}                           
\newcommand{\smoothA}[1]{\alpha_{#1}}             
\newcommand{\smoothB}[1]{\beta_{#1}}              
\newcommand{\crossT}[1]{\chi_{#1}}                
\newcommand{\blift}{A^{\sharp}}                   

\title{Counting spanning quasi-trees of ribbon graphs:\\
determinants and \#P-completeness}
\author{William Whistler\\
\texttt{will@whistler.uk}}
\date{}

\begin{document}
\maketitle

\begin{abstract}
A quasi-tree of a connected ribbon graph is a spanning ribbon
subgraph with exactly one boundary component; quasi-trees play the
role of spanning trees in the topological graph theory of embedded
graphs. We prove that counting them is \#P-complete under
polynomial-time Turing reductions, already for bouquets. The proof identifies every
nonempty framed chord diagram, up to natural identifications,
with a $4$-regular map equipped with a distinguished A-trail, in
such a way that quasi-trees correspond to A-trails, whose counting
is \#P-complete by a theorem of Ge and \v{S}tefankovi\v{c}.
Through the framed Cohn--Lempel equality the count is also an
interlace-polynomial evaluation --- $\ip(\lcg;2,1)$, the number of
full-rank induced subgraphs of the looped circle graph $\lcg$ of
the diagram --- placing it on the line $y=1$ left open in the
complexity classification of Bl\"aser and Hoffmann; a cloning
argument then makes every fixed rational point of that line,
other than the trivial $(1,1)$, \#P-hard on looped circle graphs,
even when a framed chord representation is supplied. On the tractable side, the same
$\GFtwo$ model yields short proofs of the known determinantal
cases: for orientable ribbon graphs the count is a determinant,
essentially the Matrix--Quasi-tree Theorem of Merino, Moffatt and
Noble, proved here via Bouchet's principal unimodularity, and for
bouquets with exactly one non-orientable loop it is a sum of two
orientable determinants, equivalent by a rank-one determinant
identity to the determinant formula of Deng, Jin and Yan. A
concluding conjecture locates the boundary of these
determinantal methods at regularity of the lifted delta-matroid,
where the torsor theory of Baker, Ding and Kim then applies.


\end{abstract}

\tableofcontents

\section{Introduction}\label{sec:intro}

Kirchhoff's matrix--tree theorem is the prototype of a rare
phenomenon: an exact counting problem solved by a determinant. In
the complexity classification that began with Valiant \cite{Val79},
such problems are the exception, and the contrast between
determinantal and \#P-complete counting is a recurring theme. This
paper exhibits both sides of that contrast for the spanning
quasi-trees of ribbon graphs: important orientable and
near-orientable classes are determinantal, whereas the unrestricted
counting problem is \#P-complete.

A ribbon graph is a surface with boundary presented
combinatorially: vertex discs joined by bands, each band possibly
half-twisted --- equivalently, a graph cellularly embedded in a
surface, orientable or not. A \emph{quasi-tree} is a spanning
ribbon subgraph with exactly one boundary component; when the
surface is the plane, quasi-trees are spanning trees. The term was
coined by Dasbach, Futer, Kalfagianni, Lin and Stoltzfus
\cite[Def.~3.1]{DFKLS10} in the wake of their work relating the
Jones polynomial to ribbon graphs \cite{DFKLS08}; the
Bollob\'as--Riordan polynomial admits an expansion indexed by
quasi-trees
\cite{BR01,BR02,CKS11}, they are the feasible sets of the
ribbon-graphic delta-matroid (delta-matroids being Bouchet's
\cite{Bou87b}; the ribbon-graphic case is \cite{CMNR19}), and
partial duality \cite{Chm09} acts on them by symmetric difference
\cite[Thm~5.1]{CMNR19} (see, e.g., \cite{EMM13} for
background).

On the tractable side, Merino, Moffatt and Noble \cite{MMN25}
proved a Matrix--Quasi-tree Theorem: the number of quasi-trees of
an orientable ribbon graph is a determinant. Deng, Jin and Yan
\cite{DJY24} extended the determinant to bouquets with exactly one
non-orientable loop, and Ding and Kim \cite{DK26} to the strictly
larger class of pseudo-orientable ribbon graphs; Deng, Jin, Yan and
Yan \cite{Deng25} give a symbolic determinant-and-reduction
formula encoding all quasi-trees of an arbitrary bouquet, and
Merino \cite{Mer23} computes the count for fans and wheels. Baker, Ding
and Kim \cite{BDK25} refine the critical-group picture to a torsor
structure, which T\'othm\'er\'esz \cite{Tot26} identifies with
the sandpile action of the medial digraph. We are not aware of a
hardness result in this literature.

The present paper's viewpoint is that quasi-tree counting is an
\emph{interlace polynomial} evaluation; the placement brings it
into the complexity theory of that polynomial, and identifies the
previously unresolved evaluation at which its unrestricted
complexity is then determined. The interlace polynomials of Arratia,
Bollob\'as and Sorkin \cite{ABS04a,ABS04b} arose from counting
Eulerian circuits in the service of DNA sequencing; Traldi
\cite{Tra13} interpreted them through circuit partitions of
$4$-regular graphs; and Bl\"aser and Hoffmann \cite{BH08} mapped
the complexity of evaluating the two-variable polynomial
$\ip(x,y)$, proving \#P-hardness almost everywhere --- with the
line $y=1$ left open. On the trail side, Kotzig \cite{Kot68b}
showed that A-trails of plane $4$-regular maps are counted by
spanning trees; Andersen, Bouchet and Jackson \cite{ABJ96}
developed the structure theory of A-trails on surfaces of low
genus, and Yan and Jin \cite{YJ22} on general surfaces, connecting
them to twisted duals and quasi-tree bouquets; and counting Eulerian circuits is \#P-complete in general
\cite{BW05}, while Ge and \v{S}tefankovi\v{c} \cite{GS12} proved
both Eulerian-circuit counting in $4$-regular graphs and A-trail
counting in $4$-regular maps \#P-complete.

We prove four results. Throughout, a \emph{framed chord diagram} is
a bouquet presented combinatorially --- chords on a circle, one
twist bit per chord --- and $\lcg$ is its looped circle graph:
vertices are chords, edges are interlacements, loops mark twists.
For a connected ribbon graph $G$ --- in particular a bouquet, or
a framed chord diagram --- the object counted is
\[
\qtn(G)\coloneqq
\#\bigl\{S\subseteq E(G):
(V(G),S)\ \text{has exactly one boundary component}\bigr\},
\]
the spanning ribbon subgraph on $S$ being written $(V(G),S)$, and
the counting problem takes such a $G$ as input; formal encodings
are fixed in Section~\ref{sec:atrails}.

\begin{theoA}[hardness; \S\ref{sec:atrails}]
Computing $\qtn$ is \#P-complete under polynomial-time Turing
reductions, already for bouquets. Equivalently: evaluating
$\ip(\lcg;2,1)=\ipn(\lcg;1)$ --- counting the induced subgraphs of
$\lcg$ whose adjacency matrix is nonsingular over $\GFtwo$ --- is
\#P-complete for looped circle graphs, even when a framed chord
representation is supplied. Consequently, for every fixed rational
$\xi\neq 1$, evaluating $\ip(\lcg;\xi,1)$ is \#P-hard on the same
class, under polynomial-time Turing reductions.
\end{theoA}

\begin{theoB}[the A-trail dictionary; \S\ref{sec:atrails}]
For every connected $4$-regular map $\fourmap$ with a chosen
A-trail $\atrail$, one can construct in polynomial time a framed
chord diagram $D(\fourmap,\atrail)$ whose states correspond
bijectively to the smooth transition systems of $\fourmap$,
preserving the number of circuits; consequently
$\qtn(D(\fourmap,\atrail))$ is the number of A-trails of
$\fourmap$. Conversely, every nonempty framed chord diagram arises
from such a pair, up to independent reversal of the cyclic order
at each vertex and cyclic shift or reversal of $\atrail$. An A-trail of a connected $4$-regular
map always exists and can be constructed in polynomial time.
\end{theoB}

\begin{theoC}[orientable determinant; \S\ref{sec:determinants}]
The number of spanning quasi-trees of a connected orientable
ribbon graph is computable in polynomial time; for a bouquet it
equals
$\det(I+\skl)$, where $\skl$ is the first-endpoint principally
unimodular skew lift of the interlacement matrix defined in
Section~\ref{subsec:orientable}.
\end{theoC}

\begin{theoD}[one non-orientable loop; \S\ref{sec:determinants}]
Let $D$ be a framed chord diagram of size $n$ --- equivalently, a
bouquet --- whose unique twisted chord is $e_{0}$, and let $D_{0}$
be $D$ with $e_{0}$ untwisted. Then
\[
\qtn(D)=\qtn(D_{0})+\qtn(D_{0}-e_{0})
=\det\bigl(I_{n}+\skl(D_{0})\bigr)
+\det\bigl(I_{n-1}+\skl(D_{0}-e_{0})\bigr),
\]
computable in polynomial time.
\end{theoD}

The contributions have the following hierarchy.
Theorem~A, with its line corollary, is the paper's main new
result. Theorem~B is the structural bridge behind it: a statewise
A-trail specialisation of Traldi's circuit-partition
interpretation \cite{Tra11,Tra13}, whose additional ingredient is
the local choice of rotations making the two permitted transitions
at each vertex the two smooth ones. Theorems C and~D are short
recoveries, from the same $\GFtwo$ model, of known determinantal
results of \cite{MMN25} and \cite{DJY24} respectively; they
establish the tractable side of the contrast in the title. The
detailed relations to the literature --- the directed dictionary
and its determinant counts (Bouchet, Macris--Pul\'e, Lauri,
Merino), the twisted-dual analysis of Yan and Jin, and the
complexity map of Bl\"aser and Hoffmann --- are discussed where
they arise, in
Sections~\ref{sec:determinants}--\ref{sec:remarks}; to the
best of our knowledge, \#P-completeness at $(2,1)$ and
\#P-hardness at every rational point $(\xi,1)$ with $\xi\neq 1$
had not previously appeared.

\subsection{Reading guide}\label{subsec:overview}

This paper is largely self-contained, and assumes only basic
graph theory and linear algebra over $\GFtwo$; ribbon graphs are
introduced from scratch, and no prior acquaintance with interlace
polynomials is assumed beyond the definitions quoted. In
Section~\ref{sec:tracing} we define framed chord diagrams and
corner tracing, prove the tracing lemmas, and reduce general
ribbon graphs to diagrams. In Section~\ref{sec:gftwo} we convert
tracing into $\GFtwo$-linear algebra through the framed
Cohn--Lempel theorem, and identify the quasi-tree count as the
interlace-polynomial evaluation $\ip(\lcg;2,1)$. In
Section~\ref{sec:determinants} we prove the determinantal
Theorems C and~D. In Section~\ref{sec:atrails} we prove the
A-trail dictionary (Theorem~B) and the hardness Theorem~A, with
its consequence for the line $y=1$. Readers principally
interested in the complexity results may take
Theorem~\ref{theo:framedcl} as the boundary-counting input and
pass directly from Section~\ref{subsec:interlace} to
Section~\ref{sec:atrails}. Section~\ref{sec:remarks} closes
with remarks on the plane case, a conjecture delimiting the
bouquets whose quasi-trees a determinant can detect, and the
torsor structures that the theory of Baker, Ding and Kim
induces on that conjectural class.

\section{Ribbon graphs, chord diagrams, and boundary
tracing}\label{sec:tracing}

The purpose of this section is twofold: to define framed chord
diagrams and the corner tracing that computes their boundary
components, and to prove the small toolkit of tracing lemmas ---
alternation, letter tracking, even crossing, chord deletion, handle
attachment, and substitution --- on which the whole paper runs. We
close with the reduction from general ribbon graphs.

\subsection{Framed chord diagrams and corner
tracing}\label{subsec:diagrams}

Intuitively, a framed chord diagram presents a \emph{bouquet}: a
ribbon graph with one vertex disc and $n$ bands attached to its
boundary, the band $c$ carrying a half-twist exactly when its twist
bit $t_{c}$ is $1$. Flattening the vertex disc to a circle and
drawing each band as a chord gives the combinatorial object we work
with. Formally, a \emph{framed chord diagram} $D$ of size $n$ is a
partition of the \emph{position set} $\{0,1,\dotsc,2n-1\}$ of an
oriented circle into $n$ unordered pairs, the \emph{chords},
together with a \emph{twist bit} $t_{c}\in\{0,1\}$ for each chord
$c$. Two chords $c=\{a,b\}$ and $d=\{a',b'\}$ \emph{interlace} if
exactly one of $a',b'$ lies strictly between $a$ and $b$ on the
circle. When $n=0$ the diagram is empty; this degenerate case is
load-bearing below, not a curiosity. We call a chord subset
$S\subseteq E(D)$ a \emph{state}.

Boundary components are computed by \emph{corner tracing}.
For a state $S$, form the $2$-regular \emph{corner graph} on the
$4n$ \emph{corners} $\corner{p}{L},\corner{p}{R}$ for positions
$p$:
\begin{itemize}
\item \emph{arcs}: $\corner{p}{R}$---$\corner{p+1\bmod 2n}{L}$ for
  every position $p$;
\item \emph{band edges}, for each chord $c=\{a,b\}\in S$:
  $\corner{a}{L}$---$\corner{b}{R}$ and
  $\corner{a}{R}$---$\corner{b}{L}$ if $t_{c}=0$ (\emph{straight});
  $\corner{a}{L}$---$\corner{b}{L}$ and
  $\corner{a}{R}$---$\corner{b}{R}$ if $t_{c}=1$ (\emph{crossed});
\item \emph{pass edges}, for each chord $c=\{a,b\}\notin S$:
  $\corner{a}{L}$---$\corner{a}{R}$ and
  $\corner{b}{L}$---$\corner{b}{R}$.
\end{itemize}
The three wirings are shown in Figure~\ref{fig:wirings}; from this
point on we draw arcs dotted and band or pass edges solid.
The number of cycles of the corner graph is $\bc(S)$, the number of
\emph{boundary components} of the state; for the empty diagram we
set $\bc(\varnothing)\coloneqq 1$, the core circle being its own
boundary cycle. A state $S$ is a \emph{quasi-tree} if $\bc(S)=1$,
and the object of study is
\[
\qtn(D)\coloneqq\#\{S\subseteq E(D):\bc(S)=1\}.
\]

\begin{figure}
\centering
\begin{tikzpicture}[
  corner/.style={circle,fill,inner sep=1.1pt},
  wiring/.style={thick},
  arcedge/.style={densely dotted},
  lab/.style={font=\small},
  clab/.style={font=\scriptsize}]
\begin{scope}[shift={(0,0)}]
  \node[font=\small] at (1.1,1.9) {pass ($c\notin S$)};
  \node[corner,label=left:{\scriptsize $(a,\mathrm{L})$}]  (aL) at (0,1.2) {};
  \node[corner,label=right:{\scriptsize $(a,\mathrm{R})$}] (aR) at (2.2,1.2) {};
  \node[corner,label=left:{\scriptsize $(b,\mathrm{L})$}]  (bL) at (0,0) {};
  \node[corner,label=right:{\scriptsize $(b,\mathrm{R})$}] (bR) at (2.2,0) {};
  \draw[wiring] (aL) -- (aR);
  \draw[wiring] (bL) -- (bR);
  \draw[arcedge] (aL) -- ++(-0.45,0.3);
  \draw[arcedge] (aR) -- ++(0.45,0.3);
  \draw[arcedge] (bL) -- ++(-0.45,-0.3);
  \draw[arcedge] (bR) -- ++(0.45,-0.3);
\end{scope}
\begin{scope}[shift={(4.6,0)}]
  \node[font=\small] at (1.1,1.9) {straight ($c\in S$, $t_c=0$)};
  \node[corner,label=left:{\scriptsize $(a,\mathrm{L})$}]  (aL) at (0,1.2) {};
  \node[corner,label=right:{\scriptsize $(a,\mathrm{R})$}] (aR) at (2.2,1.2) {};
  \node[corner,label=left:{\scriptsize $(b,\mathrm{L})$}]  (bL) at (0,0) {};
  \node[corner,label=right:{\scriptsize $(b,\mathrm{R})$}] (bR) at (2.2,0) {};
  \draw[wiring] (aL) -- (bR);
  \draw[wiring] (aR) -- (bL);
  \draw[arcedge] (aL) -- ++(-0.45,0.3);
  \draw[arcedge] (aR) -- ++(0.45,0.3);
  \draw[arcedge] (bL) -- ++(-0.45,-0.3);
  \draw[arcedge] (bR) -- ++(0.45,-0.3);
\end{scope}
\begin{scope}[shift={(9.2,0)}]
  \node[font=\small] at (1.1,1.9) {crossed ($c\in S$, $t_c=1$)};
  \node[corner,label=left:{\scriptsize $(a,\mathrm{L})$}]  (aL) at (0,1.2) {};
  \node[corner,label=right:{\scriptsize $(a,\mathrm{R})$}] (aR) at (2.2,1.2) {};
  \node[corner,label=left:{\scriptsize $(b,\mathrm{L})$}]  (bL) at (0,0) {};
  \node[corner,label=right:{\scriptsize $(b,\mathrm{R})$}] (bR) at (2.2,0) {};
  \draw[wiring] (aL) -- (bL);
  \draw[wiring] (aR) -- (bR);
  \draw[arcedge] (aL) -- ++(-0.45,0.3);
  \draw[arcedge] (aR) -- ++(0.45,0.3);
  \draw[arcedge] (bL) -- ++(-0.45,-0.3);
  \draw[arcedge] (bR) -- ++(0.45,-0.3);
\end{scope}
\end{tikzpicture}
\caption{The three wirings at the two positions of a chord.}
\label{fig:wirings}
\end{figure}
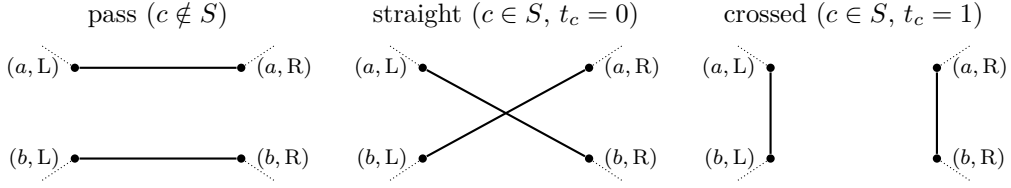

The definition is best absorbed through an example, which will
recur throughout the paper. Let $D$ be the diagram of size $2$ with
chords $c_{0}=\{0,2\}$, twisted, and $c_{1}=\{1,3\}$, untwisted
(Figure~\ref{fig:example}). Tracing the four states by hand:
$\bc(\emptyset)=1$ (the four pass edges and four arcs form one
cycle, the core circle); $\bc(\{c_{0}\})=1$ (a M\"obius band has
one boundary curve); $\bc(\{c_{1}\})=2$ (an untwisted band splits
the boundary); and $\bc(\{c_{0},c_{1}\})=1$. The quasi-trees are
$\emptyset$, $\{c_{0}\}$ and $\{c_{0},c_{1}\}$, so $\qtn(D)=3$.

\begin{figure}
\centering
\begin{tikzpicture}[
  posn/.style={circle,fill,inner sep=1.2pt},
  lab/.style={font=\small},
  clab/.style={font=\scriptsize}]
\begin{scope}
  \draw (0,0) circle (1.3);
  \node[posn,label=above:{\small $0$}] (p0) at (0,1.3) {};
  \node[posn,label=right:{\small $1$}] (p1) at (1.3,0) {};
  \node[posn,label=below:{\small $2$}] (p2) at (0,-1.3) {};
  \node[posn,label=left:{\small $3$}]  (p3) at (-1.3,0) {};
  \draw[thick] (p0) -- (p2)
    node[pos=0.28,sloped,font=\scriptsize,allow upside down]{$\times$}
    node[pos=0.72,right=1pt,font=\scriptsize]{$c_0$};
  \draw[thick] (p1) -- (p3) node[pos=0.25,above=1pt,font=\scriptsize]{$c_1$};
\end{scope}
\begin{scope}[shift={(4.2,0)}]
  \node[posn,label=below:{\small $c_0$}] (v0) at (0,0) {};
  \node[posn,label=below:{\small $c_1$}] (v1) at (1.6,0) {};
  \draw[thick] (v0) -- (v1);
  \draw[thick] (v0) to[out=120,in=60,looseness=18] (v0);
  \node[font=\scriptsize] at (0.8,-0.85) {$H$: loop at $c_0$ (twist), edge (interlacement)};
\end{scope}
\end{tikzpicture}
\caption{The running example and its looped circle graph.}
\label{fig:example}
\end{figure}
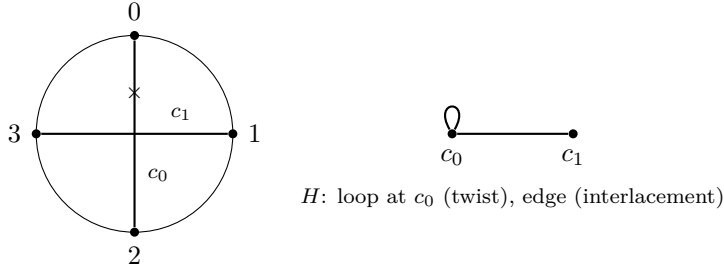

The \emph{Euler genus} of $D$ is
$\eg(D)\coloneqq n+1-\bc(E(D))$, and $D$ is \emph{orientable} if
every twist bit is $0$. We shall see in
Corollary~\ref{coro:genusrank} that $\eg(D)$ is a matrix rank.

\subsection{The tracing lemmas}\label{subsec:tracinglemmas}

Each lemma of this subsection does one job. Throughout, a
\emph{walk} in a corner graph traverses its edges, and a
\emph{segment} is a walk obtained from a cycle by removing edges.

\begin{lemm}[alternation and parity]\label{lemm:alternation}
Every corner is incident with exactly one arc and one non-arc edge.
Consequently every cycle of every corner graph alternates arcs with
non-arc edges and has even length, and every open segment obtained
by removing two non-arc edges from cycles begins and ends with arcs
and has odd length.
\end{lemm}

\begin{proof}
The corner $\corner{p}{L}$ meets the arc from $\corner{p-1}{R}$ and
the unique pass or band edge at its position; the corner
$\corner{p}{R}$ meets the arc to $\corner{p+1}{L}$ and its pass or
band edge.
\end{proof}

\begin{lemm}[letter tracking]\label{lemm:letters}
Assign to the corner $\corner{p}{s}$ the letter
$s\in\{\mathrm{L},\mathrm{R}\}$. Arcs, pass edges and straight band
edges join corners with opposite letters; crossed band edges join
corners with equal letters. Consequently, a walk segment $\gamma$
starts and ends at equal letters if and only if
$|\gamma|+\crx{\gamma}$ is even, where $\crx{\gamma}$ is the number
of crossed edges traversed.
\end{lemm}

\begin{proof}
Inspection of the wirings in Figure~\ref{fig:wirings}; the parity
statement follows edge by edge.
\end{proof}

\begin{lemm}[even crossing]\label{lemm:evencrossing}
Every cycle of every corner graph traverses an even number of
crossed edges.
\end{lemm}

\begin{proof}
A cycle $\Gamma$ has even length by Lemma~\ref{lemm:alternation}
and returns to its starting corner, hence to its starting letter;
by Lemma~\ref{lemm:letters}, $|\Gamma|+\crx{\Gamma}$ is even, so
$\crx{\Gamma}$ is even.
\end{proof}

\begin{lemm}[chord deletion]\label{lemm:deletion}
Let $c$ be a chord of $D$ and let $D-c$ be the diagram obtained by
deleting $c$ and its two positions. Then $\bc_{D}(S)=\bc_{D-c}(S)$
for every state $S\subseteq E(D)\setminus\{c\}$.
\end{lemm}

\begin{proof}
Suppose first that $c$ is the sole chord. Then $D-c$ is the empty
diagram, with $\bc(\varnothing)=1$ by convention, while the corner
graph of $S=\emptyset$ in $D$ is the single cycle alternating the
two pass edges of $c$ with the two arcs; both sides equal $1$.

Otherwise, in the corner graph of $S$ in $D$ the four corners at
the positions of $c$ each have degree two, met by one arc and one
pass edge (Lemma~\ref{lemm:alternation}). Delete the two pass
edges of $c$ together with those four corners, and concatenate the
arcs incident with each deleted position into one edge. The result
is the corner graph of $S$ in $D-c$: the concatenations produce
exactly the arcs of $D-c$ between the surviving positions ---
including the case of cyclically adjacent positions of $c$, where
three arcs concatenate into one --- and every band and pass edge
of the chords of $D-c$ is untouched. Suppressing degree-two
vertices in this way neither merges nor splits any cycle.
\end{proof}

The chord-deletion lemma is used silently whenever we write
$\qtn(D-c)$ for the number of quasi-trees of $D$ avoiding $c$; we
shall not remark on it again.

\begin{lemm}[handle attachment]\label{lemm:handle}
Let $S$ be a state and $c\notin S$ a chord with pass edges at
positions $a$ and $b$.
\begin{enumerate}
\item If the two pass edges lie on different cycles of the corner
  graph, then $\bc(S\cup c)=\bc(S)-1$, for either twist bit.
\item If they lie on one cycle $\Gamma$, let $\gamma$ be either
  segment of $\Gamma$ cut at the two pass edges. Then
  $\bc(S\cup c)=\bc(S)+1$ if $\crx{\gamma}+t_{c}$ is even, and
  $\bc(S\cup c)=\bc(S)$ if $\crx{\gamma}+t_{c}$ is odd; the
  criterion does not depend on the choice of segment, since
  $\crx{\Gamma}$ is even (Lemma~\ref{lemm:evencrossing}).
\end{enumerate}
\end{lemm}

\begin{proof}
Adjoining $c$ replaces its two pass edges by two band edges on the
same four corners
$\corner{a}{L},\corner{a}{R},\corner{b}{L},\corner{b}{R}$, so only
the cycles through those corners change.

(1) Cutting each of two cycles once leaves two open segments.
Either wiring uses one new edge from each $a$-side corner to a
$b$-side corner, so any traversal alternates the two segments: one
cycle results.

(2) Cutting $\Gamma$ at the two pass edges leaves segments $\gamma$
and $\gamma'$ whose endpoint pairs are
$\{\corner{a}{X},\corner{b}{Y}\}$ and
$\{\corner{a}{\overline{X}},\corner{b}{\overline{Y}}\}$ for some
letters $X,Y$, writing $\overline{\mathrm{L}}=\mathrm{R}$ and
$\overline{\mathrm{R}}=\mathrm{L}$. By Lemma~\ref{lemm:alternation} the segment $\gamma$ begins
and ends with arcs and has odd length, so by
Lemma~\ref{lemm:letters}, $X=Y$ if and only if $\crx{\gamma}$ is
odd. The straight wiring joins opposite letters; hence $\gamma$
closes onto itself precisely if $X\neq Y$, that is, if
$\crx{\gamma}$ is even, and then $\gamma'$ closes too (its
endpoints carry the complementary letters) and $\bc$ increases by
one; otherwise each new edge joins $\gamma$ to $\gamma'$ and $\bc$
is unchanged. The crossed wiring joins equal letters and
interchanges the two outcomes.
\end{proof}

\subsection{Ribbon graphs, substitution, and the
reduction}\label{subsec:reduction}

The same tracing computes boundary components of an arbitrary
ribbon graph $G$: a collection of vertex discs, each with cyclically
ordered attachment positions on its boundary, and bands joining
pairs of positions, each band carrying a twist bit. The twist bits
of a multi-vertex ribbon graph are relative data: choose an
orientation of every vertex disc, and each band is straight or
crossed with respect to those choices. Reversing a disc toggles
the bit of every band with exactly one attachment on that disc,
while a loop band based there, both of whose attachments reverse,
is unchanged; the resulting signed rotation system presents the
same ribbon graph. The corner
graph of a spanning-subgraph state has arcs running along each disc
boundary between consecutive positions, band edges for the bands of
the state (straight or crossed by the twist), and pass edges at the
attachments of absent bands; an isolated disc --- one with no
positions --- contributes one cycle, its whole boundary, in
accordance with the convention $\bc(\varnothing)=1$ above. To
police the vocabulary: ribbon graphs have \emph{bands}, diagrams
have \emph{chords}, and in both cases a \emph{state} is a spanning
selection; a bouquet is a ribbon graph with one disc, and its
diagram presentation is as above. Ribbon graphs are henceforth
connected unless stated otherwise. This costs no generality: no
spanning state of a disconnected ribbon graph has a single
boundary component, and under the componentwise convention common
in the literature (a quasi-tree in each component; see, e.g.,
\cite{DK26}) the counts multiply over components, so disconnected
inputs reduce to connected ones. The lemmas of
Section~\ref{subsec:tracinglemmas} extend, with the same proofs,
to states of ribbon graphs, the isolated-disc convention covering
a deletion that removes the last attachment position of a disc.

\begin{lemm}[substitution]\label{lemm:substitution}
Let $G$ be a ribbon graph and $S_{0}$ a state with $\bc(S_{0})=1$.
There is a framed chord diagram $D'$, computable in polynomial time
from $G$ and $S_{0}$, whose chords are the bands of $G$ and which
satisfies
\[
\bc_{D'}(U)=\bc_{G}(S_{0}\mathbin{\triangle}U)
\qquad\text{for every }U\subseteq E(G).
\]
In particular $\qtn(D')=\qtn(G)$.
\end{lemm}

\begin{proof}
The circle of $D'$ is the single boundary cycle of $S_{0}$, with
either of its two orientations chosen to orient the core circle
--- the other choice reflects the diagram; the
chord of a band $e$ is attached at the two places where that cycle
traverses $e$'s data --- its two pass edges if $e\notin S_{0}$, its
two band edges if $e\in S_{0}$. Formally, the boundary cycle
traverses exactly $2|E(G)|$ non-arc edges, two for each band;
these traversals, in cyclic order, are the positions of $D'$, and
the boundary segments between consecutive ones are its arcs. The
twist bit of the chord is calibrated by
\[
t'_{e}\coloneqq
\begin{cases}
1 & \text{if } \bc_{G}(S_{0}\mathbin{\triangle}\{e\})=1,\\
0 & \text{otherwise.}
\end{cases}
\]
Fix a band $e$. In the corner graph of $S_{0}$ it is represented by
two designated edges; toggling its membership replaces exactly
those two edges by two edges on the same four endpoints, in one of
the two possible ways --- straight or crossed relative to the
cycle --- and which way is a fixed local wiring datum, independent
of the rest of the state. In $D'$, toggling the chord of $e$
replaces its two pass edges on the circle --- the same four
endpoints, under the identification of the circle of $D'$ with the
boundary cycle of $S_{0}$ --- again in one of the two ways, fixed
by $t'_{e}$. The two wirings are distinguished by the single
calibration state: rewiring one cycle in the two possible ways
yields one cycle and two cycles respectively, being the two splices
of a cut cycle as in the proof of Lemma~\ref{lemm:handle}(2).
Since $t'_{e}$ is defined by that calibration, the wirings agree
for every band; the replacements for distinct bands act on
disjoint quadruples of corners, so they may be performed
simultaneously, and for every $U$ the corner graph of $U$ in $D'$
is canonically isomorphic --- after suppressing the degree-two
subdivisions introduced by reading boundary segments as core arcs
--- to the corner graph of $S_{0}\mathbin{\triangle}U$ in $G$.
The count is preserved because $U\mapsto S_{0}\mathbin{\triangle}U$
is a bijection on states preserving $\bc$.
\end{proof}

We show the construction in Figure~\ref{fig:substitution}, on the
smallest interesting input: the doubled edge, reduced at a
spanning tree, yields the untwisted crossing pair --- an example
to which we return after Corollary~\ref{coro:genusrank}. The
construction is the boundary-walk description of the partial
dual $G^{S_{0}}$, presented as a bouquet (see \cite{Chm09,EMM13});
we shall not use this identification, but it explains the company
the lemma keeps.

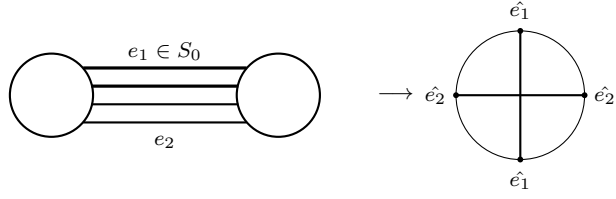
\begin{figure}
\centering
\begin{tikzpicture}[
  every node/.style={font=\scriptsize},
  band/.style={thick},
  treeband/.style={very thick}]
\begin{scope}
  \draw[thick] (0,0) circle (0.55);
  \draw[thick] (3.0,0) circle (0.55);
  \draw[treeband] (0.42,0.36) -- (2.58,0.36);
  \draw[treeband] (0.53,0.12) -- (2.47,0.12);
  \node at (1.5,0.58) {$e_{1}\in S_{0}$};
  \draw[band] (0.53,-0.12) -- (2.47,-0.12);
  \draw[band] (0.42,-0.36) -- (2.58,-0.36);
  \node at (1.5,-0.62) {$e_{2}$};
\end{scope}
\begin{scope}[shift={(6.2,0)}]
  \draw (0,0) circle (0.85);
  \foreach \ang/\lab in {90/{e_{1}},0/{e_{2}},270/{e_{1}},180/{e_{2}}}{
    \fill (\ang:0.85) circle (1.1pt);
    \node at (\ang:1.12) {$\hat{\lab}$};}
  \draw[thick] (90:0.85) -- (270:0.85);
  \draw[thick] (0:0.85) -- (180:0.85);
\end{scope}
\node at (4.55,0) {$\longrightarrow$};
\end{tikzpicture}
\caption{Reducing the doubled edge at a spanning tree.}
\label{fig:substitution}
\end{figure}

\begin{lemm}[reduction]\label{lemm:reduction}
Every connected ribbon graph $G$ has, computably in polynomial
time, a framed chord diagram $D'$ with $\qtn(D')=\qtn(G)$.
\end{lemm}

\begin{proof}
A spanning tree $T$ of $G$ is a state with $\bc(T)=1$:
adding its bands one by one, each addition joins a fresh disc's
boundary cycle to the cycle of the component being grown, which is
case~(1) of Lemma~\ref{lemm:handle}. Apply
Lemma~\ref{lemm:substitution} with $S_{0}=T$.
\end{proof}

It is worth stating exactly what the reduction preserves. The
substitution lemma preserves every boundary-component count after
the relabelling $U\mapsto T\mathbin{\triangle}U$ --- that is, the
whole quasi-tree set system up to symmetric difference, the
ribbon-graphic delta-matroid up to twists \cite{CMNR19} --- and it
preserves connectedness and orientability: we shall prove the
latter directly, by a parity argument, in
Observation~\ref{obse:parity}, and it is also an instance of the
orientability-preservation of partial duality
\cite[\S1.8, Lemma~(d)]{Chm09}. It need
not preserve the Euler genus, the number of vertices, or the number
and placement of twisted chords; in particular it does not carry
``exactly one non-orientable loop'' to ``exactly one twisted
chord'', which is why Theorem~D is a statement about bouquets. A
quantitative example of the genus increase follows
Corollary~\ref{coro:genusrank}.

For the remainder of Sections \ref{sec:gftwo}--\ref{sec:atrails},
$D$ denotes a framed chord diagram of size $n$.

\section{The \texorpdfstring{$\GFtwo$}{GF(2)} model and the
interlace polynomial}\label{sec:gftwo}

This section converts corner tracing into $\GFtwo$-linear
algebra, through the framed Cohn--Lempel theorem, and then places
the quasi-tree count on the complexity map of the interlace
polynomial.

\subsection{The framed interlacement matrix}\label{subsec:fim}

Subscripts $2$ on rank, nullity and determinant indicate
computation over $\GFtwo$. The \emph{framed interlacement
matrix} of $D$ is the matrix
$\fim\in\GFtwo^{n\times n}$ with $\fim_{cc}=t_{c}$ and, for
$c\neq d$, $\fim_{cd}=1$ exactly when $c$ and $d$ interlace.
Equivalently, $\fim$ is the adjacency matrix over $\GFtwo$ of the
\emph{looped circle graph} $\lcg=\lcg(D)$: the graph whose vertices
are the chords, whose edges are the interlacements, and whose loops
mark the twisted chords. For the running example,
\[
\fim=\begin{pmatrix}1&1\\1&0\end{pmatrix},
\]
the loop at $c_{0}$ providing the diagonal $1$
(Figure~\ref{fig:example}). For a state $S$ we write $\fim[S]$ for
the principal submatrix on the rows and columns of $S$, with
$\fim[\emptyset]$ the empty matrix, whose determinant is $1$ by
convention --- a convention that matches
$\bc(\varnothing)=1$ exactly.

\subsection{The framed Cohn--Lempel theorem}\label{subsec:framedcl}

The theorem we need is due, in the form below, to Traldi
\cite[Thm~4]{Tra11}, extending the classical cycle-counting
equality of Cohn and Lempel \cite[Thm~1$'$]{CL72} and the trip
matrix of Zulli \cite[Thm~2(b)(i)]{Zul95}; the framed statement
also appears, in knot-theoretic garb, in \cite{TZ09}. Traldi's
setting is that of circuit partitions of $4$-regular graphs, and
the translation to framed chord diagrams is the following lemma,
which we shall meet again, in refined form, in
Section~\ref{sec:atrails}.

Given a nonempty diagram $D$, construct a connected $4$-regular
multigraph $\fourmap_{D}$: its vertices are the chords of $D$, its
edges are the $2n$ circle arcs between cyclically consecutive
positions, each arc joining the chords owning its two endpoint
positions --- possibly one and the same chord, an arc between
cyclically adjacent positions of a single chord being a loop
whose two half-edges remain distinct. Traversing the circle traverses every arc exactly once
and visits each chord twice, through its two positions; the circle
is thus an Eulerian circuit $\atrail$ of $\fourmap_{D}$, and two
chords interlace in $D$ exactly when their two visits alternate
along $\atrail$. Each of the four half-edges at the vertex of a
chord $c=\{a,b\}$ is an arc-end at one of the positions $a,b$, and
the corners name them: reading the circle in its orientation, the
arc into position $p$ ends at $\corner{p}{L}$ and the arc out of
$p$ begins at $\corner{p}{R}$, so the corners
$\corner{a}{L},\corner{b}{L}$ are the incoming half-edges at $c$
and $\corner{a}{R},\corner{b}{R}$ the outgoing ones.

For the running example, $\fourmap_{D}$ has two vertices and four
parallel edges (Figure~\ref{fig:fd}).

\begin{figure}
\centering
\begin{tikzpicture}[every node/.style={font=\scriptsize}]
  \fill (0,0) circle (1.4pt) node[left=2pt] {$c_{0}$};
  \fill (3.0,0) circle (1.4pt) node[right=2pt] {$c_{1}$};
  \draw[thick] (0,0) to[bend left=55] node[above] {$01$} (3.0,0);
  \draw[thick] (0,0) to[bend left=20] node[above] {$12$} (3.0,0);
  \draw[thick] (0,0) to[bend right=20] node[below] {$23$} (3.0,0);
  \draw[thick] (0,0) to[bend right=55] node[below] {$30$} (3.0,0);
\end{tikzpicture}
\caption{The $4$-regular graph $\fourmap_{D}$ of the running
example.}
\label{fig:fd}
\end{figure}
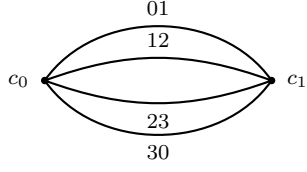

A \emph{transition} at a vertex is a partition of its four
half-edges into two pairs, and a \emph{transition system} is a
choice of one transition at each vertex; following transitions
through vertices decomposes the edge set into closed
\emph{circuits}. A state $S$ determines a transition system of
$\fourmap_{D}$: at a chord $c\notin S$, pair the half-edges as
$\atrail$ does; at a chord $c\in S$, pair them the one other way
prescribed by the band edges of the tracing.

\begin{lemm}[tracing = circuit following]\label{lemm:engine}
Let $S$ be a state of a nonempty diagram $D$.
\begin{enumerate}
\item Under the naming of half-edges by corners, the corner graph
  of $S$ is the graph on the half-edges of $\fourmap_{D}$ whose
  edges pair the two ends of each edge of $\fourmap_{D}$ (the
  arcs) and pair the half-edges at each vertex by the transition
  of $S$ (the pass and band edges). Its cycles are therefore
  exactly the circuits of the transition system of $S$, and
  $\bc(S)$ is the number of circuits.
\item In the trichotomy of \cite{Tra11}, the three wirings at a
  chord realise the three transitions as follows: the pass wiring
  follows $\atrail$; the straight wiring is the
  orientation-consistent transition that does not follow
  $\atrail$; and the crossed wiring is the
  orientation-inconsistent transition.
\end{enumerate}
\end{lemm}

\begin{proof}
(1) An arc of the corner graph joins $\corner{p}{R}$ to
$\corner{p+1}{L}$, which are precisely the two ends of the circle
arc from $p$ to $p+1$; and the non-arc edges at a chord pair its
four corners in one of the three possible ways, which under the
naming is a transition at that vertex --- for $c\notin S$ the pass
wiring, for $c\in S$ the band wiring of its twist. A circuit of
the transition system is a closed alternating sequence of edges
and transitions, that is, a cycle of this graph, and conversely;
the correspondence is a bijection.

(2) At the chord $c=\{a,b\}$ the incoming half-edges are
$\corner{a}{L},\corner{b}{L}$ and the outgoing ones
$\corner{a}{R},\corner{b}{R}$. The pass wiring pairs
$\corner{a}{L}$ with $\corner{a}{R}$, the arrival at $a$ with the
departure from $a$: it follows $\atrail$. The straight wiring
pairs $\corner{a}{L}$ with $\corner{b}{R}$ and $\corner{a}{R}$
with $\corner{b}{L}$, each incoming half-edge with the outgoing
half-edge at the \emph{other} position: a re-routing that
preserves the direction of travel, viz.\ the
orientation-consistent transition that does not follow $\atrail$.
The crossed wiring pairs $\corner{a}{L}$ with $\corner{b}{L}$ and
$\corner{a}{R}$ with $\corner{b}{R}$, the two incoming half-edges
together and the two outgoing together: the
orientation-inconsistent transition.
\end{proof}

\begin{theo}[framed Cohn--Lempel; Traldi]\label{theo:framedcl}
For every state $S$ of a framed chord diagram,
\[
\bc(S)=\nulltwo\bigl(\fim[S]\bigr)+1 .
\]
In particular, $S$ is a quasi-tree if and only if $\fim[S]$ is
nonsingular over $\GFtwo$, and
$\qtn(D)=\#\{S:\dettwo\fim[S]=1\}$.
\end{theo}

\begin{proof}
For the empty diagram both sides are $1$. Otherwise, by
Lemma~\ref{lemm:engine}, $\bc(S)$ is the number of circuits of the
transition system of $S$ on the connected $4$-regular graph
$\fourmap_{D}$, and that system deviates from $\atrail$ exactly on
$S$, orientation-inconsistently exactly at the twisted chords of
$S$. Traldi's matrix $I_{P}$ for this circuit partition
\cite[Thm~4 and Cor.~5]{Tra11} is the interlacement matrix of
$\atrail$ restricted to the deviating vertices, with diagonal ones
at the orientation-inconsistent ones --- which is precisely
$\fim[S]$, interlacement along $\atrail$ being chord interlacement.
His extended Cohn--Lempel equality gives
$\#\text{circuits}=\nulltwo(I_{P})+c(\fourmap_{D})
=\nulltwo(\fim[S])+1$.
\end{proof}

Restated in delta-matroid terms, Theorem~\ref{theo:framedcl}
says that $\fim$ is a binary representation of the ribbon-graphic
delta-matroid of the bouquet, the quasi-trees being its feasible
sets; that representation is a theorem of Bouchet, and the
looped-interlacement description of it is spelled out in
\cite[Thm~5.44 and \S5.7]{CMNR19} and
\cite[Prop.~2.21]{DK26}. The Cohn--Lempel route proves
it from scratch in the form this paper uses.

\begin{coro}[genus is a rank]\label{coro:genusrank}
$\eg(D)=\rktwo(\fim)$.
\end{coro}

\begin{proof}
Theorem~\ref{theo:framedcl} at $S=E(D)$:
$\eg(D)=n+1-\bc(E(D))=n-\nulltwo(\fim)=\rktwo(\fim)$.
\end{proof}

The reduction of Lemma~\ref{lemm:reduction} does not preserve this
quantity, and the failure is unbounded. The smallest example is the
doubled edge: two vertex discs joined by two parallel untwisted
bands, a plane ribbon graph of Euler genus $0$ with $\qtn=2$.
Reducing at a spanning tree yields the diagram of two interlaced
untwisted chords (Figure~\ref{fig:substitution}) --- the torus
bouquet, of Euler genus $2$. More
generally, in the plane graph of $k$ triangles sharing one vertex,
every spanning tree $T$ admits a spanning tree $T'$ using a
different edge pair in each triangle, so that
$U\coloneqq T\mathbin{\triangle}T'$ has $|U|=2k$. By
Lemma~\ref{lemm:substitution},
$\bc_{D'}(U)=\bc_{G}(T')=1$, so $\fim[U]$ is nonsingular by
Theorem~\ref{theo:framedcl}, and the diagram $D'$ produced from
$T$ has Euler genus at least $2k$ by
Corollary~\ref{coro:genusrank}.

\begin{coro}[parity]\label{coro:parity}
$\qtn(D)\equiv\dettwo(I+\fim)\pmod 2$.
\end{coro}

\begin{proof}
Expanding $\dettwo(xI+\fim)$ as a sum over principal minors and
evaluating at $x=1$ gives
$\dettwo(I+\fim)=\sum_{S}\dettwo\fim[S]$ in $\GFtwo$, which is the
mod-$2$ count of Theorem~\ref{theo:framedcl}.
\end{proof}

\begin{coro}[zero column]\label{coro:zerocolumn}
If a chord $e$ is untwisted and interlaces no other chord, then no
quasi-tree contains $e$.
\end{coro}

\begin{proof}
For $e\in S$ the matrix $\fim[S]$ has a zero row.
\end{proof}

\subsection{The interlace-polynomial
identity}\label{subsec:interlace}

Recall the two-variable interlace polynomial of Arratia,
Bollob\'as and Sorkin \cite{ABS04b}: for a graph $\lcg$ with loops
allowed, with adjacency matrix over $\GFtwo$,
\[
\ip(\lcg;x,y)
=\sum_{S\subseteq V(\lcg)}
(x-1)^{\rktwo \fim[S]}\,(y-1)^{\nulltwo \fim[S]},
\]
and the vertex-nullity interlace polynomial is
$\ipn(\lcg;y)=\ip(\lcg;2,y)$ \cite[\S4]{ABS04b}. At the point
$(2,1)$, with the convention $0^{0}=1$, the rank factor is trivial
and the nullity factor kills every singular principal minor, so
$\ip(\lcg;2,1)$ counts the induced subgraphs whose adjacency
matrix is nonsingular --- an observation already made in
\cite[\S5]{ABS04b}. Combining this
expansion with Theorem~\ref{theo:framedcl} gives the identity on
which the paper turns.

\begin{theo}[the model identity]\label{theo:modelidentity}
For every framed chord diagram $D$ with looped circle graph
$\lcg$,
\[
\qtn(D)
=\#\{S:\dettwo\fim[S]=1\}
=\ipn(\lcg;1)
=\ip(\lcg;2,1).
\]
\end{theo}

The identity places quasi-tree counting on the line $y=1$ of the
complexity map of Bl\"aser and Hoffmann \cite{BH08}, to which we
return in Section~\ref{sec:atrails}. For the running example, the
three nonsingular principal minors of $\fim$ --- the empty one,
$(1)$ at $c_{0}$, and $\fim$ itself --- match the three traced
quasi-trees, and $\ip(\lcg;2,1)=3$.

\section{The determinantal cases}\label{sec:determinants}

We now prove the two determinantal theorems: Theorem~C by
principal unimodularity, and Theorem~D from it as a corollary.
Both results are known ---
Theorem~C is essentially the Matrix--Quasi-tree Theorem of Merino,
Moffatt and Noble \cite{MMN25}, and Theorem~D is due to Deng, Jin
and Yan \cite{DJY24} --- and the section's contribution is a pair
of short proofs from the $\GFtwo$ model.

\subsection{Skew lifts and the orientable
determinant}\label{subsec:orientable}

Let $D$ be orientable, so that $\fim$ has zero diagonal. Order the
chords by their first endpoints and let the \emph{skew lift}
$\skl\in\mathbb{Z}^{n\times n}$ have $\skl_{cd}=1$ and
$\skl_{dc}=-1$ whenever $c<d$ interlace, all other entries zero.
Bouchet proved that such a matrix is \emph{principally unimodular}:
every principal minor lies in $\{-1,0,1\}$
\cite[Cor.~2]{Bou87}. In his terminology circle graphs are
alternance graphs, principal unimodularity is the unimodularity of
an orientation, and the orientations induced by an arbitrary
choice of orientation for each chord are principally unimodular;
our first-endpoint lift is the instance in which every
chord is oriented from its first endpoint to its second (see also
the explicit construction in \cite[\S2]{BCG98}). Concretely, if
$c<d$ interlace then the circle meets their endpoints in the order
(first of $c$, first of $d$, second of $c$, second of $d$): both
marked endpoints lead their chords in this alternation, which is
exactly the case in which Bouchet's rule directs the edge from
$c$ to $d$, giving $\skl_{cd}=+1$ as displayed.

\begin{theo}[orientable determinant]\label{theo:orientable}
For an orientable framed chord diagram,
$\qtn(D)=\det(I+\skl)$.
\end{theo}

\begin{proof}
Expanding over principal minors,
$\det(I+\skl)=\sum_{S}\det\skl[S]$. Each $\skl[S]$ is
skew-symmetric, so
$\det\skl[S]=\operatorname{Pf}(\skl[S])^{2}\ge 0$, the Pfaffian
being zero in odd sizes, and principal unimodularity puts the
determinant in $\{0,1\}$. Reducing
modulo $2$ turns $\det\skl[S]$ into $\dettwo\fim[S]$, so
$\det\skl[S]$ is exactly the quasi-tree indicator of $S$
(Theorem~\ref{theo:framedcl}), and the sum is $\qtn(D)$.
\end{proof}

To extend the theorem to ribbon graphs we need the reduction of
Lemma~\ref{lemm:reduction} to respect orientability, which it
does for a parity reason worth recording.

\begin{obse}[orientable parity]\label{obse:parity}
If $G$ is an orientable ribbon graph with $v$ vertex discs, then
every state $S$ satisfies $\bc(S)\equiv v+|S| \pmod 2$.
\end{obse}

\begin{proof}
The state $S$, as a surface with boundary, has Euler
characteristic $v-|S|$, and its Euler genus --- the sum over its
connected components, equal to twice the number of components
minus $(v-|S|)$ minus $\bc(S)$ --- is even, every component being
orientable when $G$ is. Twice the number of components being even
too, $\bc(S)\equiv v+|S|\pmod 2$.
\end{proof}

Equivalently: the delta-matroid of an orientable ribbon graph is
even. That statement is classical --- see
\cite[Prop.~5.3(4)]{CMNR19}, who trace it to Bouchet --- and the
parity proof above keeps the section self-contained. (The rank
identity of Corollary~\ref{coro:genusrank} likewise has a
delta-matroid reading; cf.\ \cite[Prop.~5.3(3)]{CMNR19}.)

\begin{coro}[Theorem C]\label{coro:theoA}
The number of spanning quasi-trees of an orientable ribbon graph
is computable in polynomial time.
\end{coro}

\begin{proof}
Reduce at a spanning tree $T$: the diagram $D'$ of
Lemma~\ref{lemm:substitution} has twist bits
$t'_{e}=1$ exactly when $\bc_{G}(T\mathbin{\triangle}\{e\})=1$;
but by Observation~\ref{obse:parity},
$\bc_{G}(T\mathbin{\triangle}\{e\})\equiv
v+|T|+1\equiv\bc_{G}(T)+1\equiv 0\pmod 2$, toggling one band
changing $|S|$ by one. So every twist bit
vanishes, $D'$ is orientable, and
Theorem~\ref{theo:orientable} computes
$\qtn(G)=\qtn(D')=\det\bigl(I+\skl(D')\bigr)$ in polynomial time.
\end{proof}

The relation of Theorem~\ref{theo:orientable} to the
Matrix--Quasi-tree Theorem of Merino, Moffatt and Noble
\cite[Thms~6.3 and~6.8]{MMN25} deserves a precise statement. For
a bouquet, the first-endpoint lift $\skl$ is, up to simultaneous
permutation and diagonal $\pm 1$-switching --- a signed
permutation congruence --- their directed-bouquet matrix, which they identify with Bouchet's signed
circle-graph matrix; Theorem~\ref{theo:orientable} is thus the
unit-weight bouquet case of their weighted Matrix--Quasi-tree
Theorem. The determinant has since been carried to a further
level of generality by Baker, Ding and Kim: the bases of any
regular orthogonal representation are counted by $\det(I+A)$
\cite[Cor.~2.9]{BDK25}, a result they note generalizes the
Matrix--Quasi-tree Theorem, with the critical-group
presentation recovered as their equation~(5); for a bouquet,
$\skl$ agrees with their boundary-word matrix at the empty
quasi-tree up to the sign and orientation conventions
catalogued in \cite[Rems.~A.3 and~A.4]{BDK25}, and, like ours,
their route into these determinants is Bouchet's circle-graph
unimodularity \cite{Bou87}. The published proof of Merino,
Moffatt and Noble, like ours, proceeds through
principally unimodular principal minors, and they record the
directed-medial derivation through the BEST theorem as an
alternative route; the general-map statement follows by applying
the substitution of Lemma~\ref{lemm:substitution} at a reference
quasi-tree. We have included the argument because the framed
Cohn--Lempel formulation places the determinant in the model used
throughout this paper, and claim nothing beyond that.

\subsection{One non-orientable loop}\label{subsec:oneloop}

Now let $D$ have exactly one twisted chord $e_{0}$, and let
$D_{0}$ be $D$ with $e_{0}$ untwisted. Note that $D-e_{0}$ and
$D_{0}-e_{0}$ are the same diagram, and both $D_{0}$ and
$D_{0}-e_{0}$ are orientable. Theorem~D rests on a pointwise
identity of quasi-tree indicators, for which we write
$\mathbf{1}_{D}(S)$ for the indicator of $\bc_{D}(S)=1$.

\begin{lemm}[pointwise identity]\label{lemm:pointwise}
For every state $S\subseteq E(D)\setminus\{e_{0}\}$,
\[
\mathbf{1}_{D}(S\cup e_{0})
=\mathbf{1}_{D_{0}}(S\cup e_{0})+\mathbf{1}_{D_{0}}(S),
\]
and the two summands on the right are never both $1$.
\end{lemm}

\begin{proof}
The corner graphs of $S$ in $D$ and $D_{0}$ coincide, and $S$
contains only untwisted chords, so no crossed edges occur in
them. Consider the two pass edges of $e_{0}$ in the corner graph
of $S$.

If they lie on different cycles, then adjoining $e_{0}$ merges
those cycles for either twist bit
(Lemma~\ref{lemm:handle}(1)), so
$\mathbf{1}_{D}(S\cup e_{0})=\mathbf{1}_{D_{0}}(S\cup e_{0})$;
moreover $\bc(S)\ge 2$, so $\mathbf{1}_{D_{0}}(S)=0$ and the
identity holds.

If they lie on one cycle, then every segment $\gamma$ has
$\crx{\gamma}=0$, so by Lemma~\ref{lemm:handle}(2) the twisted
$e_{0}$ leaves $\bc$ unchanged while the untwisted $e_{0}$
increases it by one. If $\bc(S)=1$ this gives
$\mathbf{1}_{D}(S\cup e_{0})=1$,
$\mathbf{1}_{D_{0}}(S\cup e_{0})=0$ and
$\mathbf{1}_{D_{0}}(S)=1$; if $\bc(S)\ge 2$, all three indicators
vanish. In every case the identity holds, and the two summands
are never both $1$: the second requires $\bc(S)=1$, which forces
the one-cycle case, in which the first vanishes.
\end{proof}

\begin{coro}[Theorem D]\label{coro:theoB}
If $e_{0}$ is the unique twisted chord of $D$, then
\[
\qtn(D)=\qtn(D_{0})+\qtn(D_{0}-e_{0})
=\det\bigl(I_{n}+\skl(D_{0})\bigr)
+\det\bigl(I_{n-1}+\skl(D_{0}-e_{0})\bigr),
\]
computable in polynomial time.
\end{coro}

\begin{proof}
Summing the pointwise identity over all
$S\subseteq E(D)\setminus\{e_{0}\}$ counts, on the left, the
quasi-trees of $D$ containing $e_{0}$, and on the right the
quasi-trees of $D_{0}$ containing $e_{0}$ plus those avoiding it
--- that is, $\qtn(D_{0})$. The quasi-trees of $D$ avoiding
$e_{0}$ number $\qtn(D-e_{0})=\qtn(D_{0}-e_{0})$
(Lemma~\ref{lemm:deletion}; the two diagrams are equal). Adding
the two counts gives the first equality, and
Theorem~\ref{theo:orientable} applies to both orientable terms.
\end{proof}

The result is due to Deng, Jin and Yan \cite[Thm~4.3]{DJY24}, in
the form of the single integer determinant
$\det(I_{n}+A_{\mathrm{DJY}})$ of the signed matrix
$A_{\mathrm{DJY}}\coloneqq K+uu^{\mathsf{T}}$, whose skew part is
$K$ and whose one nonzero diagonal entry, at $e_{0}$, is supplied
by the coordinate vector $u$ of $e_{0}$. Their construction begins
with an arbitrary labelling of the chords and an arbitrary
ordering of the two endpoints of each chord; after indexing the
chords by first endpoint and ordering each endpoint pair from
first to second, the skew part is exactly $K=\skl(D_{0})$. Under
arbitrary labels and endpoint orderings the two matrices are
related by a signed permutation congruence
$K=Q^{\mathsf{T}}\skl(D_{0})Q$, with the coordinate vector of
$e_{0}$ transforming accordingly, and every determinant below is
unchanged. The universal
adjugate form of the rank-one determinant identity,
\[
\det(B+uu^{\mathsf{T}})
=\det B+u^{\mathsf{T}}\!\operatorname{adj}(B)\,u ,
\]
valid with no invertibility assumption, at $B=I_{n}+K$ evaluates
$u^{\mathsf{T}}\!\operatorname{adj}(B)\,u$ to the
$(e_{0},e_{0})$-cofactor
$\det\bigl((I_{n}+K)[\overline{e_{0}}]\bigr)$, writing
$\overline{e_{0}}$ for the index set $E(D)\setminus\{e_{0}\}$; since
$K[\overline{e_{0}}]=\skl(D_{0}-e_{0})$ with the inherited chord
order, their determinant is exactly our sum. Nor is
the proof above really different from theirs: their argument runs
through the partial Petrial and the parity of feasible sets of an
even delta-matroid, of which Lemma~\ref{lemm:pointwise} is the
corner-tracing transposition. We claim no novelty for
Theorem~D beyond, perhaps, the brevity of this route.

The running example traces the corollary: its unique twisted
chord is $e_{0}=c_{0}$, the untwisted double crossing $D_{0}$ has
$\qtn(D_{0})=\det\begin{pmatrix}1&1\\-1&1\end{pmatrix}=2$, the
single-chord diagram $D_{0}-e_{0}$ has
$\qtn=\det(1)=1$, and indeed $3=2+1$.

\FloatBarrier
\section{A-trails and \#P-completeness}\label{sec:atrails}

The two remaining theorems are proved in this section: the
dictionary (Theorem~B), by which framed chord diagrams are, up to
natural identifications, exactly $4$-regular maps with a
distinguished A-trail and quasi-trees correspond to A-trails; and
the hardness theorem (Theorem~A), with its consequence for the
line $y=1$ of the interlace polynomial.

\subsection{Maps, smooth transitions, and
A-trails}\label{subsec:maps}

A \emph{$4$-regular map} is a $4$-regular multigraph $\fourmap$
--- loops and parallel edges permitted, connectedness not assumed
--- with a \emph{rotation} at each vertex: a cyclic order of its
four half-edges. Following \cite{GS12} we work with rotation systems
only; surfaces play no part in what follows. Transition systems
and their circuits are as in Section~\ref{subsec:framedcl}. At a
vertex $v$ with rotation $(h_{1},h_{2},h_{3},h_{4})$, a transition
is \emph{smooth} if it pairs only half-edges adjacent in the
rotation; there are exactly two smooth transitions,
$\smoothA{v}=\{h_{1}h_{2}\mid h_{3}h_{4}\}$ and
$\smoothB{v}=\{h_{2}h_{3}\mid h_{4}h_{1}\}$, and one
\emph{crossing} transition
$\crossT{v}=\{h_{1}h_{3}\mid h_{2}h_{4}\}$
(Figure~\ref{fig:transitions}). The terminology is that of
\cite{YJ22}. Reversing a rotation preserves the trichotomy, and
nothing below distinguishes a rotation from its reversal.

\begin{figure}
\centering
\begin{tikzpicture}[
  stub/.style={thick},
  pairing/.style={very thick},
  every node/.style={font=\scriptsize}]
\foreach \name/\xshift in {alpha/0,beta/4.2,chi/8.4}{
\begin{scope}[shift={(\xshift,0)}]
  \draw (0,0) circle (0.9);
  \draw[stub] (45:0.9)  -- (45:1.45)  node[pos=1.25]{$h_{1}$};
  \draw[stub] (-45:0.9) -- (-45:1.45) node[pos=1.25]{$h_{2}$};
  \draw[stub] (225:0.9) -- (225:1.45) node[pos=1.25]{$h_{3}$};
  \draw[stub] (135:0.9) -- (135:1.45) node[pos=1.25]{$h_{4}$};
\end{scope}}
\begin{scope}[shift={(0,0)}]
  \draw[pairing] (45:0.9) to[bend right=50] (-45:0.9);
  \draw[pairing] (135:0.9) to[bend left=50] (225:0.9);
  \node at (0,-1.75) {$\smoothA{v}=\{h_{1}h_{2}\mid h_{3}h_{4}\}$};
\end{scope}
\begin{scope}[shift={(4.2,0)}]
  \draw[pairing] (-45:0.9) to[bend right=50] (225:0.9);
  \draw[pairing] (135:0.9) to[bend right=50] (45:0.9);
  \node at (0,-1.75) {$\smoothB{v}=\{h_{2}h_{3}\mid h_{4}h_{1}\}$};
\end{scope}
\begin{scope}[shift={(8.4,0)}]
  \draw[pairing] (45:0.9) -- (225:0.9);
  \draw[pairing] (-45:0.9) -- (135:0.9);
  \node at (0,-1.75) {$\crossT{v}=\{h_{1}h_{3}\mid h_{2}h_{4}\}$};
\end{scope}
\end{tikzpicture}
\caption{The two smooth transitions and the crossing transition.}
\label{fig:transitions}
\end{figure}
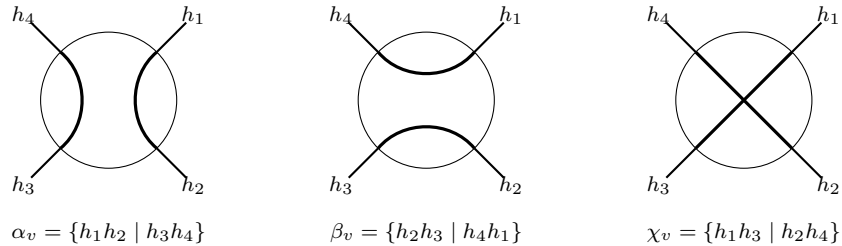

An \emph{A-trail} of $\fourmap$ is a smooth transition system with
exactly one circuit; a map with an A-trail is necessarily
connected, one circuit covering all edges. The definition fixes a
counting convention --- an A-trail is what is elsewhere an
undirected Eulerian circuit, taken up to cyclic shift and
reversal, all of whose transitions are smooth --- and the next
lemma reconciles it exactly with the sequence-based conventions.

A \emph{dart} is an edge together with a choice of direction; the
two directions of a loop give distinct darts.

\begin{lemm}[counting conventions]\label{lemm:conventions}
Let $T$ be a one-circuit transition system on a connected
$4$-regular map with $m$ edges. Then $T$ determines exactly $2m$
dart-rooted directed Eulerian sequences; cyclic shift partitions
them into two orbits of size $m$, one for each direction around
the circuit, and reversal interchanges the two orbits.
Consequently the counts of one-circuit systems, of directed
Eulerian circuits up to cyclic shift, and of dart-rooted directed
Eulerian sequences stand in the ratios $1:2:2m$, and \#P-hardness
under polynomial-time Turing reductions is indifferent to the
choice among these conventions.
\end{lemm}

\begin{proof}
From each starting dart, following the continuation prescribed by
$T$ at every arrival produces a directed Eulerian sequence
returning to its start, since $T$ has one circuit; distinct
starting darts give distinct rooted sequences, their initial darts
differing, so there are exactly $2m$. Fixing a direction around
the circuit, the $m$ darts traversed in that direction are the $m$
possible starting points, and the corresponding rooted sequences
are precisely the cyclic shifts of one another: one orbit of size
$m$. The other direction traverses exactly the $m$ opposite darts,
giving the second orbit, and reversing a sequence starts it at a
reversed dart, so reversal interchanges the two orbits. Every
sequence reads back $T$ from its consecutive pairs at each vertex,
so the quotients by nothing, by cyclic shift, and by shift and
reversal together count $2m$, $2$ and $1$ per system. A factor
computable from the input does not affect \#P-hardness under
Turing reductions.
\end{proof}

Ge and \v{S}tefankovi\v{c} \cite{GS12} present A-trails as
vertex--edge sequences, and Brightwell and Winkler \cite{BW05}
count Eulerian circuits directed and unrooted;
Lemma~\ref{lemm:conventions} converts the sequence convention to
the transition-system convention by an explicitly computable
factor, so their hardness theorem applies to our count.

\begin{lemm}[existence; Yan--Jin]\label{lemm:existence}
Every connected $4$-regular map has an A-trail, and one can be
constructed in polynomial time.
\end{lemm}

\begin{proof}
Choose either smooth transition at each vertex, and let the
resulting smooth system have $r$ circuits. If $r\ge 2$, some
vertex $v$ is visited by two distinct circuits: otherwise the edge
set of any one circuit would meet no other circuit at any vertex
and would thus be a union of connected components, contradicting
connectedness. The two passages through $v$ use the two pairs of
its current smooth transition and belong to different circuits;
switching $v$ to its other smooth transition re-pairs the four
half-edges so that each new pair takes one half-edge from each old
pair, and the splice argument of Lemma~\ref{lemm:handle}(1) merges
the two circuits into one, leaving all other circuits unchanged.
Each switch decreases the circuit count by one, so at most $r-1$
switches produce an A-trail, and every step is a local
recomputation.
\end{proof}

The statement and the switching proof are Lemma~3.1 of Yan and Jin
\cite{YJ22}; the algorithmic reading is immediate from their
argument, and we include the proof to keep the section
self-contained. Related compatible-tour results go back to Kotzig
\cite{Kot68a,Kot68b} and Jackson \cite{Jac87}.

\subsection{The dictionary}\label{subsec:dictionary}

\begin{theo}[the A-trail dictionary]\label{theo:dictionary}
Let $\fourmap$ be a connected $4$-regular map with $n$ vertices
and $\atrail$ an A-trail of $\fourmap$. There is a framed chord
diagram $D(\fourmap,\atrail)$ of size $n$, computable in
polynomial time, whose states correspond bijectively to the smooth
transition systems of $\fourmap$ --- the state $S$ corresponding
to the system deviating from $\atrail$ exactly at the vertices in
$S$ --- with $\bc(S)$ equal to the number of circuits.
In particular,
\[
\qtn\bigl(D(\fourmap,\atrail)\bigr)
=\#\{\text{A-trails of }\fourmap\}.
\]
\end{theo}

\begin{proof}
Fix an arbitrary starting point and direction of traversal for
$\atrail$; other choices cyclically shift or reverse the resulting
diagram. The circle of $D(\fourmap,\atrail)$ is the passage
sequence of $\atrail$: cyclically, the trail makes exactly two
passages through each vertex, giving $2n$ positions, and the chord
of a vertex joins its two passages. The trail traverses each edge
of $\fourmap$ exactly once, between consecutive passages --- a
loop edge, its two half-edges distinct, between two consecutive
passages through its own vertex --- so the construction of
Section~\ref{subsec:framedcl} applied to this diagram returns
$\fourmap$ with $\atrail$ as its distinguished Eulerian circuit,
the edges of $\fourmap$ being the circle arcs.

It remains to calibrate the twist bits. At a vertex $v$ the trail's
own transition is smooth, since $\atrail$ is an A-trail; of the
two remaining transitions, exactly one is the other smooth
transition and one is the crossing $\crossT{v}$. By
Lemma~\ref{lemm:engine}(2) the two band wirings of the chord of
$v$ realise the orientation-consistent and the
orientation-inconsistent re-routings, and we set $t_{v}\coloneqq 0$
if the orientation-consistent re-routing is the smooth one, and
$t_{v}\coloneqq 1$ otherwise --- a local inspection of the
rotation at $v$ against the two passages of $\atrail$.

The calibration is illustrated in Figure~\ref{fig:calibration}.
With it, the transition system of a state $S$
(Section~\ref{subsec:framedcl}) follows $\atrail$ off $S$ and uses
the other \emph{smooth} transition at each vertex of $S$: it is a
smooth transition system deviating from $\atrail$ exactly on $S$.
Since a smooth system is determined by its deviation set, the
correspondence is a bijection, and
$\bc(S)$ counts its circuits by Lemma~\ref{lemm:engine}(1).
Specialising to one circuit gives the count of A-trails.
\end{proof}

\begin{figure}
\centering
\begin{tikzpicture}[
  stub/.style={thick},
  pairing/.style={very thick},
  every node/.style={font=\scriptsize}]
\foreach \xshift in {0,4.6,9.2}{
\begin{scope}[shift={(\xshift,0)}]
  \draw (0,0) circle (0.9);
  \draw[stub] (45:0.9)  -- (45:1.4);
  \draw[stub] (-45:0.9) -- (-45:1.4);
  \draw[stub] (225:0.9) -- (225:1.4);
  \draw[stub] (135:0.9) -- (135:1.4);
\end{scope}}
\begin{scope}[shift={(0,0)}]
  \draw[pairing] (45:0.9) to[bend right=50] (-45:0.9);
  \draw[pairing] (135:0.9) to[bend left=50] (225:0.9);
  \node[align=center,text width=3.4cm] at (0,-1.9)
    {used by $\atrail$: pass (chord absent)};
\end{scope}
\begin{scope}[shift={(4.6,0)}]
  \draw[pairing] (-45:0.9) to[bend right=50] (225:0.9);
  \draw[pairing] (135:0.9) to[bend right=50] (45:0.9);
  \node[align=center,text width=3.6cm] at (0,-1.9)
    {other smooth transition: chord present, $t_{v}$ calibrated};
\end{scope}
\begin{scope}[shift={(9.2,0)}]
  \draw[pairing] (45:0.9) -- (225:0.9);
  \draw[pairing] (-45:0.9) -- (135:0.9);
  \node[align=center,text width=3.6cm] at (0,-1.9)
    {crossing $\crossT{v}$: excluded from the state choices};
\end{scope}
\end{tikzpicture}
\caption{The twist calibration at one vertex of
$(\fourmap,\atrail)$.}
\label{fig:calibration}
\end{figure}
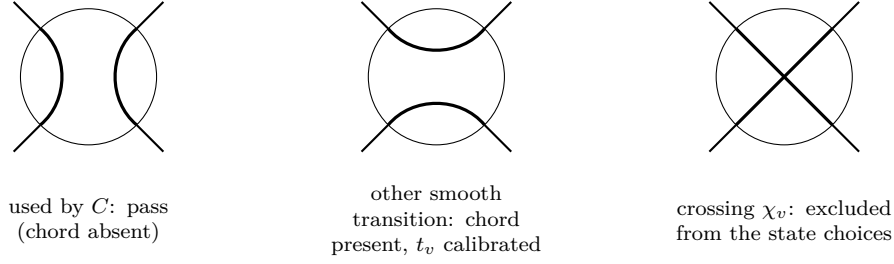

\begin{obse}[the converse]\label{obse:converse}
Every nonempty framed chord diagram $D$ arises as
$D(\fourmap,\atrail)$. At each vertex of the graph $\fourmap_{D}$
of Section~\ref{subsec:framedcl}, the pass wiring and the band
wiring of its chord realise two distinct transitions; and any two
distinct transitions of a four-element half-edge set are the two
smooth transitions of some cyclic order, unique up to reversal,
the remaining transition being its crossing --- place the two
pairs of that remaining transition so that they alternate. Give
$\fourmap_{D}$ these rotations and take $\atrail$ to be the
circle: then $\atrail$ uses only smooth transitions, hence is an
A-trail, and the construction of Theorem~\ref{theo:dictionary}
returns $D$ itself. The two constructions are mutually inverse up
to independent reversal of the cyclic order at each vertex and
cyclic shift or reversal of the distinguished trail, none of which
the diagram records. The empty
diagram is excluded: its circle visits no vertex, and it
corresponds to no $4$-regular map.
\end{obse}

It is worth saying precisely whose machinery this is. Traldi's
circuit-partition interpretation \cite{Tra11,Tra13} already
identifies $\ipn$ of a looped circle graph with a generating
function for compatible circuit partitions of the associated
$4$-regular graph --- in our setting, Theorem~\ref{theo:framedcl}
combined with the expansion of
Section~\ref{subsec:interlace} --- and
Theorem~\ref{theo:dictionary} is its statewise A-trail
specialisation: the additional local observation is the choice of
rotations that makes the two permitted transitions at each vertex
precisely the two smooth ones, together with the converse
construction. The directed counterpart --- quasi-trees of an
\emph{orientable} bouquet are the Eulerian circuits of an
associated $2$-in $2$-out digraph --- is implicit in Bouchet
\cite{Bou89} and explicit in Merino \cite[Thm~13]{Mer23}, the
determinant enumeration of Eulerian circuits in such digraphs
being due to Macris and Pul\'e \cite{MP96} and Lauri
\cite{Lau97}, and it lives on the tractable side: directed Eulerian circuits are counted by the BEST
theorem, which gives an alternative directed-medial derivation of
the Matrix--Quasi-tree determinant, recorded in
\cite[Rem.~3.7]{MMN25}. The standard directed construction does
not extend directly to arbitrary twists, and the undirected
specialisation lands, as we now show, on the hard side.
The closest structural relative is the twisted-dual analysis of
Yan and Jin \cite{YJ22}, whose proof of their Theorem~1.1 turns a
map with a chosen A-trail, after a partial Petrial, into the
medial graph of a bouquet. On the orientable side the
correspondence now carries group structure as well:
T\'othm\'er\'esz \cite[Thms~5.2 and~5.9]{Tot26} identifies the
Jacobian of an embedded graph with the sandpile group of its
medial digraph, acting canonically on quasi-trees through the
Eulerian tours of that digraph; we return to this in
Section~\ref{sec:remarks}.

\subsection{Hardness}\label{subsec:hardness}

\begin{theo}[hardness]\label{theo:hardness}
Computing $\qtn$ is \#P-complete under polynomial-time Turing
reductions, already for framed chord diagrams, and hence for
bouquets and for ribbon graphs.
\end{theo}

\begin{proof}
Inputs are encoded plainly: a diagram by its double-occurrence
word and twist bits; a ribbon graph by a signed rotation system
--- its half-edge incidences, the cyclic attachment order at each
vertex disc, and one twist bit per band; a map by its graph and
rotations. Membership: a state $S$ is a
polynomially checkable certificate, by corner tracing or by
Gaussian elimination over $\GFtwo$ (Theorem~\ref{theo:framedcl}).

Hardness: Ge and \v{S}tefankovi\v{c} prove that counting
Eulerian tours of $4$-regular graphs is polynomial-time Turing
reducible to counting A-trails of $4$-regular maps --- their
construction produces, from a $4$-regular graph $F_{0}$, a map
with exactly $2^{|V(F_{0})|}$ A-trails per Eulerian tour of
$F_{0}$ --- which together with their Theorem~1 establishes
\#P-completeness of the A-trail count \cite[Thms~1 and~3]{GS12},
building on the theorem of Brightwell and Winkler \cite{BW05}; by
Lemma~\ref{lemm:conventions} the completeness is indifferent to
the counting convention. It
therefore suffices to count A-trails of a $4$-regular map
$\fourmap$ using an oracle for $\qtn$. If $\fourmap$ is
disconnected it has no A-trail, and we answer $0$; otherwise
construct an A-trail $\atrail$ (Lemma~\ref{lemm:existence}), build
$D(\fourmap,\atrail)$, and make the single oracle query
$\qtn(D(\fourmap,\atrail))$, which is the A-trail count by
Theorem~\ref{theo:dictionary}.
\end{proof}

Note that the reduction from A-trails is exact --- one oracle
call, no interpolation --- and parsimonious at the level of
transition systems; the Turing character of the theorem is
inherited from the chain of \cite{GS12}.

\begin{coro}[equivalent formulations]\label{coro:formulations}
The following counting problems are all \#P-complete, and
polynomial-time Turing equivalent to one another:
\begin{enumerate}
\item given a framed chord diagram, the number of its quasi-trees;
\item given a connected ribbon graph, the number of its
  quasi-trees;
\item given a framed chord diagram, the number of chord subsets
  indexing a nonsingular principal submatrix of its framed
  interlacement matrix;
\item given a looped circle graph with a framed chord
  representation, the number of its full-rank induced subgraphs,
  viz.\ the evaluation $\ip(\lcg;2,1)=\ipn(\lcg;1)$;
\item given a $4$-regular map, the number of its A-trails.
\end{enumerate}
\end{coro}

\begin{proof}
Problems (1), (3) and (4) ask for the same number, by
Theorem~\ref{theo:modelidentity}; (2) reduces to (1) by
Lemma~\ref{lemm:reduction} and contains it; and (1) and (5) are
equivalent by Theorem~\ref{theo:dictionary},
Observation~\ref{obse:converse} and Lemma~\ref{lemm:existence},
the empty diagram being handled separately with count $1$ and
disconnected maps with count $0$. Completeness is
Theorem~\ref{theo:hardness}.
\end{proof}

In problem~(4) the framed chord representation is auxiliary help to the
algorithm, not a qualification weakening the statement: hardness
holds a fortiori when the input is the graph alone, promised to
be a looped circle graph, an oracle for that problem being usable
after discarding the representation.

\subsection{The line \texorpdfstring{$y=1$}{y=1} of the interlace
polynomial}\label{subsec:line}

Bl\"aser and Hoffmann \cite{BH08} proved that evaluating
$\ip(x,y)$ is \#P-hard at almost every point of the plane; their
hardness region excludes the line $y=1$ entirely, and the point
$(2,1)$ is among the explicit exceptions of their Corollary~4.4
(on the line $x=2$ the exceptional ordinates are
$0,1,2,1\pm\sqrt{2}$). Theorem~\ref{theo:hardness} supplies the
missing hard point, and their cloning technique propagates it
along the line. To keep the propagation inside our restricted
graph class we need one more construction.

\begin{lemm}[clone closure]\label{lemm:cloneclosure}
Let $D$ be a framed chord diagram with looped circle graph $\lcg$
and let $k\ge 1$. There is a diagram $D^{(k)}$, computable in
polynomial time, whose looped circle graph is the $k$-fold clone
of $\lcg$ in the sense of \cite[\S3.1]{BH08}: each vertex $v$ is
replaced by $k$ twins carrying the loop status of $v$; the twins
of $v$ are pairwise adjacent if $v$ is looped and pairwise
non-adjacent otherwise; and for distinct original vertices
$v\neq w$, every twin of $v$ is adjacent to every twin of $w$
exactly when $vw\in E(\lcg)$.
\end{lemm}

\begin{proof}
Replace each position of $D$ by a block of $k$ consecutive
positions. For a chord $c$ with positions $a,b$, attach $k$ clone
chords $c_{1},\dotsc,c_{k}$, all with twist bit $t_{c}$: clone
$c_{i}$ joins the $i$-th slot of the $a$-block to the $i$-th slot
of the $b$-block if $t_{c}=1$, and to the $(k{+}1{-}i)$-th slot if
$t_{c}=0$. Each clone has one endpoint in each block, and the
blocks occupy the intervals where the endpoints of $c$ lay, so
$c_{i}$ interlaces a clone of another chord $d$ exactly when $c$
interlaces $d$. Within a pair of blocks, equal slot orders put the
endpoints of $c_{i}$ and $c_{j}$ ($i<j$) in the alternating
pattern $a_{i}a_{j}b_{i}b_{j}$, so twisted clones pairwise
interlace; reversed slot orders nest them, so untwisted clones do
not. This is precisely the clone adjacency of \cite{BH08}.
\end{proof}

\begin{lemm}[cloning identity]\label{lemm:cloneidentity}
With $\lcg^{(k)}$ the $k$-fold clone of $\lcg$,
\[
\ip\bigl(\lcg^{(k)};x,1\bigr)
=\ip\bigl(\lcg;\,k(x-1)+1,\,1\bigr).
\]
\end{lemm}

\begin{proof}
Two clones of the same vertex have identical rows in the adjacency
matrix of $\lcg^{(k)}$: their diagonal and mutual entries agree
--- both $0$ for unlooped twins, both $1$ for looped ones --- and
their entries elsewhere agree because they have the same
neighbours. An induced subgraph containing two clones of one
vertex is therefore singular, and a nonsingular induced subgraph
selects at most one clone per vertex; selecting one clone of each
vertex of $T\subseteq V(\lcg)$ induces a matrix equal to
$\fim[T]$, and there are $k^{|T|}$ such selections. Hence, with
$0^{0}=1$,
\[
\ip\bigl(\lcg^{(k)};x,1\bigr)
=\sum_{\substack{T\subseteq V(\lcg)\\
\fim[T]\ \text{nonsingular}}}
k^{|T|}\,(x-1)^{|T|}
=\ip\bigl(\lcg;\,k(x-1)+1,\,1\bigr). \qedhere
\]
\end{proof}

This is the specialisation at $y=1$ of the cloning identity of
\cite[eq.~(3.5)]{BH08}, whose Theorem~4.12(2) is exactly the
interpolation mechanism used next.

\begin{coro}[the line $y=1$]\label{coro:line}
For every fixed rational $\xi\neq 1$, evaluating
$\ip(\lcg;\xi,1)$ is \#P-hard under polynomial-time Turing
reductions for looped circle graphs, even when a framed chord
representation is supplied.
\end{coro}

\begin{proof}
Given an oracle for $\ip(\,\cdot\,;\xi,1)$ on represented looped
circle graphs and an input diagram $D$ of size $n$, build
$D^{(k)}$ for $k=1,\dotsc,n+1$ (Lemma~\ref{lemm:cloneclosure})
and query: by Lemma~\ref{lemm:cloneidentity} the answers are the
values of $\ip(\lcg;x,1)$ at the $n+1$ distinct rational points
$k(\xi-1)+1$. Since
$\ip(\lcg;x,1)=\sum_{S:\,\nulltwo\fim[S]=0}(x-1)^{|S|}$
is a polynomial of degree at most $n$ in $x$, Lagrange
interpolation recovers it exactly, and its value at $x=2$ is
$\qtn(D)$ by Theorem~\ref{theo:modelidentity}. Oracle outputs are
exact rationals in binary, the standard fixed-evaluation model of
\cite{BH08}; for fixed rational
$\xi$ every quantity involved --- the cloned diagrams, of at most
$(n+1)n$ chords, the evaluation points, the oracle answers, and
the interpolated coefficients --- has size polynomial in the
input. Now apply Theorem~\ref{theo:hardness}.
\end{proof}

At the excluded point the evaluation is trivial:
$\ip(\lcg;1,1)=1$ for every graph, only the empty induced
subgraph having rank and nullity both zero. Thus every rational
point of the line $y=1$ other than $(1,1)$ is \#P-hard on looped
circle graphs; the same argument applies in the exact
algebraic-number evaluation model of \cite{BH08}, and we state
only the rational case. To the best of our knowledge, no
point of the line was previously known to be \#P-hard; the
propagation argument is Theorem~4.12(2) of \cite{BH08}, and the
content added here is the hard point $(2,1)$ and the closure of
the represented class under cloning.

The loops of $\lcg$ are essential to this hardness, in the
following sharp sense.

\begin{coro}[the loopless line]\label{coro:looplessline}
For an orientable framed chord diagram $D$ with circle graph
$\lcg$,
\[
\ip(\lcg;x,1)=\det\bigl(I+(x-1)\skl\bigr).
\]
Consequently, for every fixed rational $\xi$, evaluating
$\ip(\,\cdot\,;\xi,1)$ on represented loopless circle graphs is
computable in polynomial time.
\end{coro}

\begin{proof}
Expanding over principal minors with weights,
$\det(I+t\skl)=\sum_{S}t^{|S|}\det\skl[S]$, and $\det\skl[S]$ is
the quasi-tree indicator of $S$ (proof of
Theorem~\ref{theo:orientable}); and
$\ip(\lcg;x,1)=\sum_{S:\,\bc(S)=1}(x-1)^{|S|}$, by the expansion
of Section~\ref{subsec:interlace} together with
Theorem~\ref{theo:framedcl}.
\end{proof}

The identity itself is not new: through the model identity it is
the unit-weight case of the weighted Matrix--Quasi-tree Theorem
of Merino, Moffatt and Noble \cite[Thm~6.8]{MMN25}, whose
determinant --- in the form of the characteristic polynomial of
$\skl$, their Theorem~6.1 --- enumerates the quasi-trees of a
bouquet by size. What the placement on the complexity map adds
is the reading: on the line $y=1$ the loops of $\lcg$ are
exactly where the hardness lives, every fixed evaluation being a
single determinant without them and \#P-hard, outside the
trivial point, with them.

We close the section with the classical tractable island of the
A-trail formulation: for \emph{plane} $4$-regular maps the A-trail
count is polynomial, by Kotzig's bijection with spanning trees and
the matrix--tree theorem \cite{Kot68b,GS12}. This is
philosophically consistent with the plane case of quasi-tree
counting (Section~\ref{sec:remarks}); we do not claim that the two
tractable families coincide under the dictionary, the framing of
$D(\fourmap,\atrail)$ being a relation between the smooth
transitions and the chosen trail rather than a record of
planarity.

\section{Concluding remarks}\label{sec:remarks}

The classical anchor of the tractable side is the plane: a
spanning quasi-tree of a plane ribbon graph is a spanning tree, by
the Jordan curve theorem, and Kirchhoff's matrix--tree theorem
\cite{Kir1847} (see, e.g., \cite{Bol98}) counts it. Under the
dictionary of Section~\ref{sec:atrails} this sits beside Kotzig's
polynomial count of A-trails of plane $4$-regular maps
\cite{Kot68b}: two classical tractable islands on the two sides of
the correspondence, both determinantal.

The determinantal territory extends beyond Theorems C and~D: Ding
and Kim \cite{DK26} prove a determinant formula for the strictly
larger class of pseudo-orientable ribbon graphs, by methods quite
different from ours, while Theorem~\ref{theo:hardness} shows that
on unrestricted bouquets the count is \#P-complete.
Theorems C and~D succeed for a common structural reason: the
$\GFtwo$ pattern $S\mapsto\dettwo\fim[S]$ is realized over the
integers with every principal minor in $\{0,1\}$ --- for
Theorem~C by the principally unimodular skew lift $\skl$, and for
Theorem~D by the matrix $K+uu^{\mathsf{T}}$ of
Section~\ref{subsec:oneloop}, whose principal minors reproduce
the quasi-tree indicators statewise, by the rank-one identity
together with Lemma~\ref{lemm:pointwise}. Ding and Kim name the
notion at stake: a real matrix $M$, indexed by the chords,
\emph{detects} a bouquet if $\det M[S]$ is the quasi-tree
indicator of every state $S$ \cite[\S5.1]{DK26}. Theorems C
and~D construct detecting matrices; no matrix detects the
bouquets $C_{n}$, $n\ge 5$, of \cite[Prop.~5.1]{DK26}. To
formulate where we believe the boundary of detection lies, let
$t\in\GFtwo^{n}$ be the vector of twist bits and border the
framed interlacement matrix as
\[
\blift\coloneqq
\begin{pmatrix}
\fim+tt^{\mathsf{T}} & t\\
t^{\mathsf{T}} & 0
\end{pmatrix}
\in\GFtwo^{(n+1)\times(n+1)},
\]
its rows and columns indexed by the chords together with one
border index $\hat{e}$. The block $\fim+tt^{\mathsf{T}}$ has
zero diagonal, the border row records the twists, and the
adjacency between each pair of twisted chords is toggled; the
matrix $\blift$ represents the \emph{lift} of the delta-matroid
of $D$ \cite[\S2]{DK26}, a construction Ding and Kim trace to
Geelen via Murota. A \emph{PU-orientation} of $\blift$, in the
language of \cite{BCG98}, is a principally unimodular
skew-symmetric matrix with entries in $\{0,\pm 1\}$ that
reduces to $\blift$ modulo~$2$.

\begin{obse}[detection from the lift]\label{obse:detection}
If $\blift$ admits a PU-orientation, then some real matrix
detects the bouquet.
\end{obse}

\begin{proof}
Write the PU-orientation in block form and set
\[
\widetilde{\blift}=
\begin{pmatrix}
K & u\\
-u^{\mathsf{T}} & 0
\end{pmatrix},
\qquad
M\coloneqq K+uu^{\mathsf{T}},
\]
and write $\alpha(S)$ for $S$ when $|S|$ is even and for
$S\cup\{\hat{e}\}$ when $|S|$ is odd. The bordered determinant
identity of Ding and Kim \cite[Lem.~2.29]{DK26} states that
\[
\det M[S]=\det\widetilde{\blift}\bigl[\alpha(S)\bigr]
\qquad\text{for every state }S,
\]
an even-order principal minor of a principally unimodular
skew-symmetric matrix, hence a squared Pfaffian in $\{0,1\}$.
Since $M$ reduces to $\fim$ modulo $2$ (the two copies of
$tt^{\mathsf{T}}$ cancelling), this value has the parity of
$\dettwo\fim[S]$; a value in $\{0,1\}$ with the
correct parity is the quasi-tree indicator
(Theorem~\ref{theo:framedcl}), so $M$ detects the bouquet.
\end{proof}

Any detecting matrix packages the count as one determinant ---
$\qtn(D)=\det(I_{n}+M)$ --- and indeed the whole line $y=1$:
by the weighted expansion in the proof of
Corollary~\ref{coro:looplessline},
$\ip(\lcg;x,1)=\det\bigl(I_{n}+(x-1)M\bigr)$, extending that
corollary verbatim from the orientable case.

Theorems C and~D are respectively the zero- and one-twist
cases. For Theorem~C take $K=\skl$ and $u=0$. For Theorem~D
take $K=\skl(D_{0})$ and $u$ the coordinate vector of $e_{0}$:
the principal minors of $\widetilde{\blift}$ avoiding the
border are principal minors of $\skl(D_{0})$, and those through
it obey, for $R\subseteq E(D)$,
\[
\det\widetilde{\blift}\bigl[R\cup\{\hat{e}\}\bigr]=
\begin{cases}
0 & \text{if } e_{0}\notin R,\\
\det\bigl(\skl(D_{0})[R\setminus e_{0}]\bigr)
  & \text{if } e_{0}\in R,
\end{cases}
\]
the first case because the border row restricts to zero, the
second by expansion along the border row and column; odd orders
are covered automatically, both sides vanishing. Hence every
principal minor of $\widetilde{\blift}$ is a principal minor of
$\skl(D_{0})$ or zero, $\widetilde{\blift}$ is principally
unimodular, and $M$ is exactly the matrix of
Section~\ref{subsec:oneloop}.

\begin{conj}\label{conj:detect}
Conversely, whenever some real matrix detects a bouquet,
$\blift$ admits a PU-orientation --- with
Observation~\ref{obse:detection}, detection is then equivalent
to PU-orientability of the bordered lift matrix.
\end{conj}

The conjecture is consistent with what is known beyond our two
theorems: in the proof of \cite[Prop.~5.9]{DK26}, Ding and Kim
exhibit a non-pseudo-orientable bouquet for which the real
adjacency matrix of the right digraph in their Figure~14 is a
PU-orientation of precisely this lift matrix (their vertex $8$
being the border index $\hat{e}$), so
Observation~\ref{obse:detection} yields a detectable bouquet
outside the pseudo-orientable class.

The bordered matrix also admits an invariant description, in
the delta-matroid language of Ding--Kim and the
orthogonal-matroid language of Baker--Ding--Kim. Recall that an
even delta-matroid is \emph{regular} if it is representable by
a principally unimodular skew-symmetric matrix --- equivalently,
representable over every field, by a theorem of Geelen; see
\cite[Def.~2.16]{DK26} --- and that this is the same as the
corresponding orthogonal matroid --- even delta-matroids and
orthogonal matroids being equivalent encodings
\cite[fn.~5]{BDK25} --- admitting a regular orthogonal
representation \cite[\S1.4 and Def.~2.5]{BDK25}; the
tract-theoretic equivalences behind the last coincidence are
those of Jin and Kim \cite[Thm~5.2 and Ex.~3.30]{JK25}.

\begin{obse}[the invariant formulation]\label{obse:invariant}
$\blift$ admits a PU-orientation if and only if the lift of the
delta-matroid of $D$ is regular. Conjecture~\ref{conj:detect}
thus reads: a bouquet is detectable if and only if the lift of
its delta-matroid is regular.
\end{obse}

\begin{proof}
A PU-orientation represents the lift over the rationals: its
principal minors lie in $\{0,1\}$ and reduce faithfully modulo
$2$, as in the proof of Observation~\ref{obse:detection}, so
its nonsingular principal submatrices are exactly the feasible
sets of the lift. Conversely, if the lift is regular, its
orthogonal matroid admits a regular orthogonal representation;
the based matrix $A$ of that representation at the basis
corresponding to the feasible set $\varnothing$ is
skew-symmetric, principally unimodular, and $\{0,\pm 1\}$-valued
\cite[\S2.2]{BDK25}, and its support equals that of $\blift$: a
pair $\{x,y\}$ is feasible in the lift precisely when the
corresponding $2\times 2$ principal submatrix is nonsingular ---
of $A$ by \cite[Thm~2.8]{BDK25}, of $\blift$ by the $\GFtwo$
representation --- while singletons are never feasible in an
even delta-matroid and both diagonals vanish. So $A$ is a
PU-orientation of $\blift$. (The target admits no choice: a
binary even delta-matroid with $\varnothing$ feasible has a
unique binary representation \cite[\S2.2.2]{DK26}, namely
$\blift$.)
\end{proof}

Which non-orientable ribbon graphs admit such unimodular-matrix
representations is a question Kim poses explicitly in an
expository account of this circle of ideas
\cite{KimMUa,KimMUb}; for bouquets,
Conjecture~\ref{conj:detect} proposes the answer.

Regularity of the lift is exactly the hypothesis under which
the torsor theory of Baker, Ding and Kim applies.

\begin{obse}[torsors on the conjectural class]\label{obse:torsor}
Let the lift of the delta-matroid of $D$ be regular, let
$\mathcal{C}$ be a regular orthogonal representation of it, and
let $A_{\mathcal{C}}$ be the based matrix of
Observation~\ref{obse:invariant} --- a PU-orientation of
$\blift$. Then the group
$\operatorname{Jac}(\mathcal{C})=
\mathbb{Z}^{E(D)\cup\{\hat{e}\}}\big/\bigl\langle\text{rows of }
I+A_{\mathcal{C}}\bigr\rangle$
has order $\qtn(D)$ and, for each acyclic circuit signature of
$\mathcal{C}$, acts simply transitively on the quasi-trees of
$D$.
\end{obse}

\begin{proof}
The bases of the lift are the sets $\alpha(S)$ with $S$ a
quasi-tree, in bijection with the quasi-trees. Theorem~B of
\cite{BDK25}, applied to the lift with the representation
$\mathcal{C}$, gives a simply transitive action of
$\operatorname{Jac}(\mathcal{C})$ on the bases for each acyclic
circuit signature, and acyclic signatures always exist
\cite[Ex.~4.8 and Rem.~4.15]{BDK25}. The presentation of
$\operatorname{Jac}(\mathcal{C})$ by $I+A_{\mathcal{C}}$ and
the equality of its order with the number of bases are
\cite[Def.~3.1, Lem.~2.7 and Prop.~3.2]{BDK25}.
\end{proof}

Every ingredient is from \cite{BDK25}; the observation is only
that these lifts are lawful inputs to their general Theorem~B,
which their ribbon-graph results do not reach --- the lift of a
non-pseudo-orientable bouquet is not ribbon-graphic
\cite[Thm~1.1]{DK26}. Three cautions temper it. The group and
the action depend on the equivalence class of the
representation, not on the bouquet alone \cite[fn.~6]{BDK25};
no surface orientation is available to select a canonical class
or signature, so one obtains a family of torsors rather than a
canonical one; and which PU-orientations of $\blift$ arise as
based matrices we do not determine. For pseudo-orientable
bouquets the picture is cleaner: the lift is realized by an
orientable ribbon graph \cite[Thm~1.1]{DK26}, whose detecting
consequence is their Matrix--Quasi-tree Theorem
\cite[Thm~1.2]{DK26}, and the canonical action of
\cite[Thm~A]{BDK25} on that realization --- the sandpile action
of its medial digraph, by \cite[Thms~5.9 and~5.11]{Tot26} ---
transports to the quasi-trees of the original bouquet, per
realization. The example of \cite[Prop.~5.9]{DK26} shows the
regular-lift class is strictly larger.

The conjecture comes with structure on both sides.
Detectability is minor-closed \cite[Lem.~5.2]{DK26} --- chord
deletion preserves it, as does rebasing by partial duality at a
quasi-tree --- and PU-orientability is preserved under vertex
deletion and pivoting \cite[\S2]{BCG98}, so
Conjecture~\ref{conj:detect} proposes parallel obstruction
theories for bouquets and for their bordered lifts. In Ding and
Kim's family the reduction is already visible: $C_{n-2}$ is a
minor of $C_{n}$ \cite[Lem.~5.3]{DK26}, so the infinite
undetectable family reduces to the two base examples $C_{5}$
and $C_{6}$; whether detectability has finitely many
minor-minimal obstructions, or PU-orientability finitely many
pivot-minor-minimal obstructions, remains open. As for recognition, $\blift$ is precisely the
input type of the orientation-construction algorithm of
\cite[\S4]{BCG98}, which produces, in polynomial time, a
candidate signing that is principally unimodular whenever any
is; deciding PU-orientability is thereby polynomial-time
equivalent to testing principal unimodularity of a single
skew-symmetric matrix, a test whose complexity is left open
in \cite{BCG98}. The conjecture thus converts recognition of
detectability into a concrete matrix problem; it does not by
itself settle its complexity.

Theorem~\ref{theo:hardness} bounds the whole programme from
outside. The order of a finitely presented abelian group
$\mathbb{Z}^{k}/\langle\text{rows of }P\rangle$, when finite, is
computable in polynomial time --- for square $P$ it is
$|\det P|$, in general via the Smith normal form \cite{KB79} ---
so unless $\mathrm{FP}=\#\mathrm{P}$ there is no polynomial-time
construction assigning to every bouquet an integer presentation
of an abelian group of order $\qtn$. The orientability
hypotheses of \cite{MMN25,BDK25,Tot26}, the
pseudo-orientability of \cite{DK26}, and the conjectural
boundary of Conjecture~\ref{conj:detect} are in this sense
necessary features of the determinantal and torsor-theoretic
territory, not artifacts of the methods.

We note that being undetectable and being hard to count
are different properties, neither known to imply the other:
where hardness sets in below the unrestricted class --- for
instance at a fixed number of twisted chords, or at fixed Euler
genus --- remains open.

The picture to carry away is compact: for a looped circle graph
presented by chords, the full-rank evaluation $\ip(\lcg;2,1)$ is
one determinant when there are no loops, a sum of two when there
is exactly one, and \#P-complete on the unrestricted class ---
diagrams with several loops may still lie in larger determinantal
classes \cite{DK26}, Conjecture~\ref{conj:detect} proposing a
characterization of the detectable class --- with every rational
point $(\xi,1)$, $\xi\neq 1$, \#P-hard alongside it.

\section*{Acknowledgements}

Claude Fable~5 and GPT-5.6 Sol Pro were used extensively in the
development and preparation of this work.

\bibliographystyle{plain}
\bibliography{quasitree}

\end{document}